\documentclass[11pt,letterpaper,reqno]{amsart}

\usepackage{amsmath,amssymb,amsthm,amsfonts,mathtools}
\usepackage{mathrsfs}
\usepackage{xcolor}
\usepackage[T1]{fontenc}
\usepackage{microtype}
\usepackage{aliascnt}

\usepackage[colorlinks=true,
  linkcolor=blue,
  citecolor=blue,
  urlcolor=blue]{hyperref}
\usepackage[nameinlink,capitalize,noabbrev]{cleveref}

\allowdisplaybreaks
\numberwithin{equation}{section}

\newtheorem{theorem}{Theorem}[section]

\newaliascnt{problem}{theorem}

\aliascntresetthe{problem}

\newaliascnt{proposition}{theorem}
\newtheorem{proposition}[proposition]{Proposition}
\aliascntresetthe{proposition}

\newaliascnt{lemma}{theorem}
\newtheorem{lemma}[lemma]{Lemma}
\aliascntresetthe{lemma}

\newaliascnt{corollary}{theorem}

\aliascntresetthe{corollary}

\newaliascnt{remark}{theorem}
\newtheorem{remark}[remark]{Remark}
\aliascntresetthe{remark}

\crefname{theorem}{Theorem}{Theorems}
\Crefname{theorem}{Theorem}{Theorems}
\crefname{problem}{Problem}{Problems}
\Crefname{problem}{Problem}{Problems}
\crefname{proposition}{Proposition}{Propositions}
\Crefname{proposition}{Proposition}{Propositions}
\crefname{lemma}{Lemma}{Lemmas}
\Crefname{lemma}{Lemma}{Lemmas}
\crefname{corollary}{Corollary}{Corollaries}
\Crefname{corollary}{Corollary}{Corollaries}
\crefname{remark}{Remark}{Remarks}
\Crefname{remark}{Remark}{Remarks}

\newaliascnt{claim}{theorem}

\aliascntresetthe{claim}

\newaliascnt{question}{theorem}

\aliascntresetthe{question}

\newaliascnt{conjecture}{theorem}
\newtheorem{conjecture}[conjecture]{Conjecture}
\aliascntresetthe{conjecture}

\theoremstyle{definition}

\newaliascnt{definition}{theorem}

\aliascntresetthe{definition}

\newaliascnt{example}{theorem}

\aliascntresetthe{example}

\crefname{claim}{Claim}{Claims}
\Crefname{claim}{Claim}{Claims}
\crefname{question}{Question}{Questions}
\Crefname{question}{Question}{Questions}
\crefname{conjecture}{Conjecture}{Conjectures}
\Crefname{conjecture}{Conjecture}{Conjectures}
\crefname{definition}{Definition}{Definitions}
\Crefname{definition}{Definition}{Definitions}
\crefname{example}{Example}{Examples}
\Crefname{example}{Example}{Examples}

\newcommand{\R}{\mathbb R}
\newcommand{\Sph}{\mathbb S}
\newcommand{\tr}{\operatorname{tr}}
\newcommand{\HS}{\mathrm{HS}}
\newcommand{\Id}{I}
\newcommand{\dd}{\,\mathrm d}
\newcommand{\Def}{\mathfrak D}
\newcommand{\cR}{\mathcal R}
\newcommand{\cB}{\mathcal B}

\newcommand{\ip}[2]{\left\langle #1,#2\right\rangle}
\newcommand{\norm}[1]{\left\lVert #1\right\rVert}
\newcommand{\abs}[1]{\left\lvert #1\right\rvert}
\newcommand{\SobCap}{\operatorname{Cap}}

\begin{document}

\title[The Dirichlet Ashbaugh--Benguria reciprocal-gap conjecture]
{The Ashbaugh--Benguria reciprocal-gap conjecture for Dirichlet eigenvalues}

\author[Y.~Li]{Yanyang Li}
\address{School of Mathematics, Southeast University, Nanjing 211189, P. R. China}
\email{liyanyang1219@gmail.com}

\author[Q.~Tang]{Quanyu Tang}
\address{School of Mathematics and Statistics, Xi'an Jiaotong University, Xi'an 710049, P. R. China}
\email{tangquanyu827@gmail.com}

\author[H.~Zhang]{Haiqi Zhang}
\address{School of Mathematics, Shandong University, Jinan 250100, P. R. China}
\email{ZHQAQ2024@outlook.com}



\begin{abstract}
We prove the Ashbaugh--Benguria reciprocal-gap conjecture for the Dirichlet Laplacian in every dimension $N\ge2$. Specifically, if $\Omega\subset\mathbb R^N$ is a bounded domain and
$$
0<\lambda_1(\Omega)<\lambda_2(\Omega)\le\lambda_3(\Omega)\le\cdots
$$
are its Dirichlet eigenvalues, then
$$
 \sum_{i=1}^{N}
 \frac{\lambda_1(\Omega)}
 {\lambda_{i+1}(\Omega)-\lambda_1(\Omega)}
 \ge
 \frac{N}{j_{N/2,1}^2/j_{N/2-1,1}^2-1},
$$
where $j_{\mu,1}$ denotes the first positive zero of the Bessel function $J_\mu$ of the first kind of order $\mu$. We also characterize the equality case: equality holds precisely when $\Omega$ agrees with a Euclidean ball up to a set of Sobolev $H^1$-capacity zero.  In particular, among bounded Lipschitz domains, equality holds if and only if $\Omega$ is a Euclidean ball.
\end{abstract}

\maketitle


\section{Introduction}\label{sec:introduction}

\subsection{Origin of the problem}\label{subsec:origin}

Throughout the paper, by a domain we mean a connected open set. Let
$\Omega\subset\R^N$, $N\ge2$, be a bounded domain.  We consider the
Dirichlet eigenvalue problem
\[
\begin{cases}
 -\Delta u=\lambda u & \text{in }\Omega,\\
 u=0 & \text{on }\partial\Omega,
\end{cases}
\]
where the boundary condition is understood in the weak sense
$u\in H_0^1(\Omega)$. Its eigenvalues are denoted by
\[
0<\lambda_1(\Omega)<\lambda_2(\Omega)\le\lambda_3(\Omega)\le\cdots,
\]
and are repeated according to multiplicity.

Set $\nu=(N-2)/2$, and let $j_{\mu,k}$ denote the $k$th positive zero of the Bessel function $J_\mu(x)$ of the first kind of order $\mu$. The problem studied here belongs to the Payne--P\'olya--Weinberger
program on sharp low-eigenvalue isoperimetric inequalities for the
Dirichlet Laplacian. In later work on lower Dirichlet eigenvalues,
Ashbaugh and Benguria formulated the following reciprocal-gap conjecture~\cite{AshbaughBenguria1993,AshbaughBenguria1993Neumann}. The conjecture was later mentioned again in Ashbaugh's list of open problems and in Henrot's monograph~\cite{Ashbaugh1999,Henrot2006}.

\begin{conjecture}[Ashbaugh--Benguria]\label{conj:AB-Dirichlet}
Let \(N\ge2\). Then every bounded domain \(\Omega\subset\R^N\) satisfies
\begin{equation}\label{eq:conjecture-main} \sum_{i=1}^{N} \frac{\lambda_1(\Omega)}{\lambda_{i+1}(\Omega)-\lambda_1(\Omega)} \ge \frac{N}{j_{\nu+1,1}^2/j_{\nu,1}^2-1}, 
\end{equation} 
with equality if and only if $\Omega$ is a Euclidean ball.
\end{conjecture}

For arbitrary domains, the literal equality statement must be interpreted with
the usual capacitary convention for the Dirichlet Laplacian: removing a set of
Sobolev $H^1$-capacity zero does not change $H_0^1$ and hence does not change
the spectrum.  Thus the exact equality statement proved below is formulated up
to $H^1$-capacity zero; for smooth, or more generally Lipschitz, domains this
reduces to the classical statement that the only equality cases are balls.

The constant in \eqref{eq:conjecture-main} is forced by the ball. We write \(B_R(a):=\{x\in\R^N: |x-a|<R\}\) for the open Euclidean ball of radius \(R>0\) centered at \(a\in\R^N\), and write \(B_R:=B_R(0)\). By separation of variables in spherical
coordinates, together with the scaling of Dirichlet eigenvalues, see
\cite[Exercise~1.2.21 and Lemma~2.1.30]{LevitinMangoubiPolterovich2023},
one has
\[
 \lambda_1(B_R)=\frac{j_{\nu,1}^2}{R^2},
 \qquad
 \lambda_2(B_R)=\cdots=\lambda_{N+1}(B_R)
 =\frac{j_{\nu+1,1}^2}{R^2}.
\]
Indeed, the first eigenfunction is radial, while the next eigenspace is the
degree-one spherical-harmonic eigenspace, which has dimension \(N\).
Therefore
\[
 \sum_{i=1}^N
 \frac{\lambda_1(B_R)}{\lambda_{i+1}(B_R)-\lambda_1(B_R)}
 =
 \frac{N}{j_{\nu+1,1}^2/j_{\nu,1}^2-1}.
\]
Hence the right-hand side of \eqref{eq:conjecture-main} is exactly the value of the reciprocal-gap sum on balls.

\subsection{Main result}\label{subsec:main-results}

We use the standard Sobolev capacity, also called the
\(H^1\)-capacity.  For \(E\subset\R^N\), set
\[
 \SobCap(E):=
 \inf\left\{
 \int_{\R^N}\bigl(|\nabla \eta|^2+\eta^2\bigr)\,\dd x:
 \eta\in C_c^\infty(\R^N),\
 \eta\ge 1 \text{ in an open neighbourhood of } E
 \right\}.
\]
This is equivalent to the usual definition with \(C_c^1(\R^N)\) test
functions, as used for instance in
\cite[Definition~4.1.8]{LevitinMangoubiPolterovich2023}.

A statement is said to hold \emph{quasi-everywhere} if it fails only on a set
of \(H^1\)-capacity zero.  A function \(w:\R^N\to\R\) is called
\emph{quasi-continuous} if, for every \(\varepsilon>0\), there exists a set
\(E\subset\R^N\) with \(\SobCap(E)<\varepsilon\) such that
\(w|_{\R^N\setminus E}\) is continuous.  Every \(v\in H^1(\R^N)\) admits a \emph{quasi-continuous representative}, that is, a
quasi-continuous function \(\widetilde v\) such that \(\widetilde v=v\) almost
everywhere.  This representative is unique up to a set of \(H^1\)-capacity
zero.  We refer to \cite[Ch.~4]{EvansGariepy2015} for these standard facts;
see also \cite[Definitions~4.1.9--4.1.10 and Theorem~4.1.12]{LevitinMangoubiPolterovich2023}
for the terminology and the uniqueness statement.

We shall use the standard capacitary characterization of the Dirichlet
Sobolev space: identifying \(H_0^1(U)\) with its zero extensions to
\(\R^N\), one has
\begin{equation}\label{eq:cap-def-of-h01}
 H_0^1(U)=
 \{v\in H^1(\R^N): \widetilde v=0
 \text{ quasi-everywhere on } \R^N\setminus U\}.
\end{equation}
For this characterization, see
\cite[Theorem~4.1.11]{LevitinMangoubiPolterovich2023}.

In particular, if \(U,V\subset\R^N\) are bounded open sets and \(\SobCap(U\triangle V)=0\), then \(H_0^1(U)=H_0^1(V)\). Since sets of \(H^1\)-capacity zero have Lebesgue measure zero, the \(L^2\)-spaces are naturally identified. Hence, by the min--max characterization of the Dirichlet eigenvalues, the Dirichlet spectra of \(U\) and \(V\) coincide, including multiplicities.

The main result of this paper confirms \cref{conj:AB-Dirichlet}.

\begin{theorem}\label{thm:main}
Let \(N\ge2\). Then every bounded domain \(\Omega\subset\R^N\) satisfies
\begin{equation}\label{eq:main}
 \sum_{i=1}^{N}
 \frac{\lambda_1(\Omega)}{\lambda_{i+1}(\Omega)-\lambda_1(\Omega)}
 \ge
 \frac{N}{j_{\nu+1,1}^2/j_{\nu,1}^2-1}.
\end{equation}
Equality holds in
\eqref{eq:main} if and only if there exist $a\in\R^N$ and $R>0$ such that
\begin{equation}\label{eq:equality-capacity}
 \SobCap\bigl(\Omega\triangle B_R(a)\bigr)=0.
\end{equation}
In particular, if $\Omega$ is a bounded Lipschitz domain, then equality holds if and only if $\Omega$ is a Euclidean ball.
\end{theorem}

\subsection{Relation to previous work}\label{subsec:relation-prior-work}

The starting point of this circle of ideas is the Faber--Krahn inequality,
which says that the ball minimizes the first Dirichlet eigenvalue among sets
of fixed volume \cite{Faber1923,Krahn1925,Krahn1926}.  The corresponding
Neumann isoperimetric theorem is the Szeg\H{o}--Weinberger theorem: among
sets of fixed volume, the ball maximizes the first nonzero Neumann eigenvalue
\cite{Szego1954,Weinberger1956}.  For Dirichlet eigenvalue ratios and sums,
Payne, P\'olya and Weinberger proved non-sharp universal bounds and formulated
sharp ball conjectures for the first low eigenvalues
\cite{PaynePolyaWeinberger1955,PaynePolyaWeinberger1956}.  More explicitly,
writing
\[
A_N:=\frac{\lambda_2(B)}{\lambda_1(B)}
=\frac{j_{\nu+1,1}^2}{j_{\nu,1}^2}
\]
for any Euclidean ball \(B\subset\R^N\), the PPW ratio conjecture asserts that
\[
\frac{\lambda_2(\Omega)}{\lambda_1(\Omega)}\le A_N,
\]
with equality only for balls.  The stronger PPW low-eigenvalue sum conjecture
asserts that
\[
\frac{\lambda_2(\Omega)+\cdots+\lambda_{N+1}(\Omega)}
{\lambda_1(\Omega)}
\le N A_N,
\]
again with equality only for balls.  The ratio conjecture was proved by Ashbaugh and Benguria
\cite{AshbaughBenguria1991,AshbaughBenguria1992a,AshbaughBenguria1992b}.
The sum conjecture remains open in full generality; for some partial bounds
and related progress, see, for instance, \cite[(1.5)--(1.10)]{WangXia2021}.
The Bessel comparison and rearrangement techniques in the work of
Ashbaugh and Benguria are the closest classical ancestors of the present proof.

The full reciprocal-gap inequality also contains the classical PPW ratio
conjecture as an immediate consequence.  Indeed, with \(A_N\) as above, since
$\lambda_{i+1}(\Omega)\ge\lambda_2(\Omega)$ for every $i=1,\ldots,N$,
\eqref{eq:main} implies
\[
 \frac{N}{\lambda_2(\Omega)/\lambda_1(\Omega)-1}
 \ge
 \sum_{i=1}^{N}
 \frac{\lambda_1(\Omega)}{\lambda_{i+1}(\Omega)-\lambda_1(\Omega)}
 \ge
 \frac{N}{A_N - 1}.
\]
Hence $\lambda_2(\Omega)/\lambda_1(\Omega)\le A_N$, which is exactly the
Ashbaugh--Benguria theorem proving the PPW ratio conjecture.  By contrast, the
reciprocal-gap inequality does not imply the stronger PPW sum
conjecture for $\lambda_2+\cdots+\lambda_{N+1}$; it gives a sharp harmonic
mean control of the gaps, not an arithmetic mean control of the eigenvalues.

The stronger PPW low-eigenvalue sum conjecture would itself imply
\cref{conj:AB-Dirichlet}.  Indeed, suppose that
\[
 \frac{\lambda_2(\Omega)+\cdots+\lambda_{N+1}(\Omega)}{\lambda_1(\Omega)}
 \le
 N\frac{j_{\nu+1,1}^2}{j_{\nu,1}^2}.
\]
Set $x_i=\lambda_{i+1}(\Omega)/\lambda_1(\Omega)$. The conjectured sum bound gives
$N^{-1}\sum_i x_i\le A_N$.  Since $x\mapsto (x-1)^{-1}$ is convex and
decreasing on $(1,\infty)$, Jensen's inequality gives
\[
 \frac1N\sum_{i=1}^N\frac1{x_i-1}
 \ge
 \frac1{N^{-1}\sum_i x_i-1}
 \ge
 \frac1{A_N-1},
\]
which is exactly \eqref{eq:conjecture-main}.  Known universal inequalities of
Hile--Protter and Yang type provide important general constraints on low
Dirichlet eigenvalues, but they do not yield the sharp ball constant in
\cref{conj:AB-Dirichlet}; see
\cite{HileProtter1980,Ashbaugh2002,AshbaughHermi2004}. Earlier refinements of
low-eigenvalue ratio and sum bounds include
\cite{Brands1964,Chiti1982,Chiti1983,Marcellini1980,Thompson1969,ChenZheng2011}.

Wang and Xia proved the sharp estimate with one fewer Dirichlet reciprocal gap,
\begin{equation*}
 \sum_{i=1}^{N-1}
 \frac{\lambda_1(\Omega)}{\lambda_{i+1}(\Omega)-\lambda_1(\Omega)}
 \ge
 \frac{N-1}{j_{\nu+1,1}^2/j_{\nu,1}^2-1},
\end{equation*}
for bounded domains with smooth boundary, with equality only for balls
\cite{WangXia2021}. The present result completes the corresponding
\(N\)-term Dirichlet reciprocal-gap inequality and removes the boundary
regularity assumption: \cref{thm:main} applies to arbitrary bounded domains
and gives the sharp equality characterization in capacitary form.

It is useful to compare this with the Neumann problem.  Let
\[
0=\mu_0(\Omega)<\mu_1(\Omega)\le\mu_2(\Omega)\le\cdots
\]
be the Neumann eigenvalues of a smooth bounded domain $\Omega\subset\R^N$,
and let $B_\Omega$ be a ball with $|B_\Omega|=|\Omega|$.  Ashbaugh and
Benguria's Neumann reciprocal-sum conjecture asserts that
\begin{equation*}
 \sum_{i=1}^{N}\frac{1}{\mu_i(\Omega)}
 \ge
 \frac{N}{\mu_1(B_\Omega)},
\end{equation*}
with equality if and only if $\Omega$ is a ball
\cite{AshbaughBenguria1993Neumann}.  This is the natural harmonic-mean
strengthening of the Szeg\H{o}--Weinberger inequality, because the first
nonzero Neumann eigenvalue of a ball has multiplicity $N$.  The sharp
$(N-1)$-term version was proved by Xia and Wang \cite{XiaWang2023}. A recent preprint of He, Li and Tang proves the full
$N$-term Neumann conjecture for smooth bounded Euclidean domains
\cite{HeLiTang2026}. That result is
methodologically close to the matrix viewpoint adopted below, although the
Dirichlet problem has a different weight structure because all trial functions
are multiplied by the ground state.  Related developments in higher Dirichlet
ratios, constant-curvature variants, and Witten--Laplacian analogues are
treated in
\cite{AshbaughBenguria1994,AshbaughBenguria2001,BenguriaLinde2007,ChenMao2024};
see also the survey \cite{BenguriaLindeLoewe2012}.

\subsection{Proof strategy}\label{subsec:proof-strategy}

We briefly describe the main ideas of the proof.  Since the desired inequality
is invariant under dilations, we first reduce to the normalization
\[
 \lambda_1(\Omega)=\alpha^2,
 \qquad
 \alpha=j_{\nu,1},\qquad
 \beta=j_{\nu+1,1},\qquad
 C=\beta^2-\alpha^2.
\]
It is then enough to prove
\[
 \sum_{i=1}^{N}\frac{1}{\lambda_{i+1}-\lambda_1}\ge\frac{N}{C}.
\]

Let \(u\in H_0^1(\Omega)\) be chosen so that
\[
 u\ge0,\qquad -\Delta u=\lambda_1u\ \text{ weakly in }\Omega,\qquad
 \int_\Omega u^2\,\dd x=1.
\]
The trial functions are built from \(u\) and from the quotient
\[
 g(r)=\frac{J_{\nu+1}(\beta r)}{J_\nu(\alpha r)}\qquad(0<r<1),
\]
with \(g\) extended continuously to \(r=0,1\) and then extended as a constant
for \(r\ge1\).  We choose a center \(a\in\R^N\) so that, with
\[
 r=|x-a|,\qquad
 \theta=\frac{x-a}{|x-a|},\qquad
 F(x)=g(r)\theta,
\]
one has
\[
 \int_\Omega F u^2\,\dd x=0.
\]
Thus the functions \(uF_1,\ldots,uF_N\) are orthogonal to \(u\).

The proof then applies a matrix form of the Rayleigh--Ritz principle to these
\(N\) trial functions.  Define
\[
 M=\int_\Omega FF^{\top}u^2\,\dd x,\qquad
 K=\int_\Omega (DF)(DF)^{\top}u^2\,\dd x,
\]
where \(DF\) is the a.e. Jacobian matrix of \(F\).  The trace Rayleigh--Ritz
principle gives
\[
 \sum_{i=1}^{N}\frac{1}{\lambda_{i+1}-\lambda_1}
 \ge \operatorname{tr}(K^{-1}M).
\]
Therefore the main task is to prove
\[
 \operatorname{tr}(K^{-1}M)\ge \frac{N}{C}.
\]

The matrix estimate is obtained by comparing the trace-free parts of two
radial-angular matrices.  Put
\[
 P=\theta\theta^{\top},\qquad
 m(r)=g(r)^2,\qquad
 d(r)=\frac{g(r)^2}{r^2}-g'(r)^2.
\]
Then \(M=\int_\Omega m(r)P u^2\,\dd x\), while \(K\) can be rewritten using
\(P\), \(g'(r)^2\), and \(g(r)^2/r^2\).  If
\[
 G=\operatorname{tr}M,\qquad
 \Def=CG-\operatorname{tr}K,
\]
then the key estimate is that the trace-free discrepancy between the
\(m(r)P\)-matrix and the \(d(r)P\)-matrix is controlled by \(G\Def\).

This key estimate is proved by splitting the integrals into \(0<r<1\) and
\(r\ge1\).  On \(0<r<1\), a pointwise Bessel estimate for \(g\), derived from
the Riccati equation for \(rg'(r)/g(r)\), controls the radial coefficient.
The angular part is controlled by a weak estimate for the spherical averages
\[
 U(r)=\int_{\Sph^{N-1}}u(a+r\theta)^2\,\dd\sigma(\theta),
\]
which requires no regularity of \(\partial\Omega\).  On \(r\ge1\), the
constant extension of \(g\), together with the inequality \(C>N-1\), gives the
needed bound directly.

After this comparison, a completion of squares gives $\operatorname{tr}(KM)\le CG^2/N$. The matrix Cauchy--Schwarz inequality then yields
\[
 \operatorname{tr}(K^{-1}M)\operatorname{tr}(KM)\ge(\operatorname{tr}M)^2=G^2,
\]
and hence \(\operatorname{tr}(K^{-1}M)\ge N/C\).  This proves the normalized
inequality, and scaling gives the stated result.

The equality case is obtained by following the estimates backwards.  Equality
forces \(\Def=0\), which implies that the zero extension of \(u\) is supported
in a unit ball and has the first-ball eigenfunction profile there.  The
capacitary characterization of \(H_0^1(\Omega)\) then shows that \(\Omega\)
agrees with that ball up to a set of Sobolev \(H^1\)-capacity zero; for
Lipschitz domains this reduces to equality with a Euclidean ball.

\subsection{Paper organization}\label{subsec:paper-organization}

\Cref{sec:notation} fixes the Bessel constants and reduces the proof to the
normalized case \(\lambda_1(\Omega)=\alpha^2\).  \Cref{sec:bessel} defines
the ball quotient \(g\), proves the Riccati comparison estimates for
\(rg'(r)/g(r)\), derives the pointwise estimate for \(g\) used later, and
records the bound \(C>N-1\).  \Cref{sec:radial} studies the spherical averages
of the zero extension of the first eigenfunction and proves the weak radial
estimate used in the matrix argument. \Cref{sec:matrix} constructs the centered vector-valued trial map, applies the
trace Rayleigh--Ritz principle, combines the estimates from the regions
\(r<1\) and \(r\ge1\), and proves \cref{thm:main}, including the equality
characterization.  The appendix contains the algebra for the Riccati barrier
residuals.

\section{Notation and normalization}\label{sec:notation}

For $\mu>-1$, let $j_{\mu,k}$ denote the $k$th positive zero of the Bessel
function $J_\mu$.  Throughout,
\begin{equation}\label{eq:nualphabetaC}
 \nu=\frac{N-2}{2},\qquad
 \alpha=j_{\nu,1},\qquad
 \beta=j_{\nu+1,1},\qquad
 C=\beta^2-\alpha^2.
\end{equation}
For real matrices we use the Hilbert--Schmidt inner product and norm
\[
 \ip{A}{B}_{\HS}=\tr(A^{\top}B),
 \qquad
 \norm{A}_{\HS}^2=\ip{A}{A}_{\HS}.
\]
The expression in
\eqref{eq:main} is invariant under dilations.  We shall therefore scale the
set so that
\begin{equation}\label{eq:normalization}
 \lambda_1(\Omega)=\alpha^2.
\end{equation}
Under this normalization, it is enough to prove
\begin{equation}\label{eq:normalized-goal}
 \sum_{i=1}^{N}\frac{1}{\lambda_{i+1}-\lambda_1}\ge\frac{N}{C}.
\end{equation}

\section{Bessel quotient and Riccati barriers}\label{sec:bessel}

The following ball eigenfunction formulas are standard; they are recalled only to fix the normalization and the quotient used in the trial functions.

Define, for $0<r<1$,
\begin{equation}\label{eq:phirvrgr}
 \phi(r)=r^{-\nu}J_\nu(\alpha r),\qquad
 v(r)=r^{-\nu}J_{\nu+1}(\beta r),\qquad
 g(r)=\frac{v(r)}{\phi(r)}
      =\frac{J_{\nu+1}(\beta r)}{J_\nu(\alpha r)}.
\end{equation}
The functions \(\phi\) and \(v\) are understood at \(r=0\) by continuous
extension; in particular,
\[
 \phi(0)=\frac{(\alpha/2)^\nu}{\Gamma(\nu+1)}>0.
\]
These are the standard radial factors for the first Dirichlet eigenfunction
and for the degree-one Dirichlet eigenspace of the unit ball; see, for
example, the separation-of-variables discussion in \cite[Section~2]{WangXia2021}.
Since $\alpha$ and $\beta$ are first positive zeros, $\phi>0$ and $v>0$
on $(0,1)$, while both vanish simply at $r=1$.  We set $g(0)=0$, define
$g(1)$ by the limiting quotient, and extend $g$ constantly to $[1,\infty)$.
For later use set
\begin{equation}\label{eq:m-d-definitions}
 m(r)=g(r)^2\quad(r\ge0),\qquad
 d(r)=\frac{g(r)^2}{r^2}-g'(r)^2\quad(r>0).
\end{equation}
The value of $d$ at $r=0$ is irrelevant in all integrals below.

The radial equations for $\phi$ and $v$ are
\begin{align}
 -\phi''-\frac{N-1}{r}\phi'&=\alpha^2\phi,
 \label{eq:phi-ode}\\
 -v''-\frac{N-1}{r}v'+\frac{N-1}{r^2}v&=\beta^2v.
 \label{eq:v-ode}
\end{align}
Consequently,
\begin{equation}\label{eq:g-divergence-ode}
 -\bigl(r^{N-1}\phi^2g'\bigr)'
 +(N-1)r^{N-3}\phi^2g
 =Cr^{N-1}\phi^2g.
\end{equation}
The power series at the origin and the simple-zero expansions at $r=1$
give the endpoint information
\begin{equation}\label{eq:g-endpoint-asymptotics}
 g(r)=c_0r+O(r^3)\quad(r\downarrow0),\quad
 g(r)=g(1)+O((1-r)^2),\quad g'(r)=O(1-r)\quad(r\uparrow1),
\end{equation}
where $c_0>0$.  To see the expansion at $r=1$ directly, put
$\rho=1-r$.  Since $\phi$ and $v$ have simple zeros at $1$, write
\[
 \phi(r)=A\rho+B\rho^2+O(\rho^3),\qquad
 v(r)=D\rho+E\rho^2+O(\rho^3),
\]
with $A,D>0$.  From \eqref{eq:phi-ode}--\eqref{eq:v-ode},
\[
 \phi''(1)=-(N-1)\phi'(1),\qquad
 v''(1)=-(N-1)v'(1),
\]
and hence $B/A=E/D=(N-1)/2$.  Therefore the linear term in the quotient
$v/\phi$ cancels, which gives the asserted $O((1-r)^2)$ expansion and
$g'(r)=O(1-r)$.

Set
\begin{equation}\label{eq:trhr}
 t(r)=\frac{rg'(r)}{g(r)},\qquad
 h(r)=\frac{r\phi'(r)}{\phi(r)}.
\end{equation}
Dividing \eqref{eq:g-divergence-ode} by $r^{N-1}\phi^2g$ gives the Riccati
equation
\begin{equation}\label{eq:t-riccati}
 rt'+t^2+(N-2+2h)t+Cr^2-(N-1)=0,
\end{equation}
with continuous endpoint values
\begin{equation*}
 t(0)=1,
 \qquad
 t(1)=0.
\end{equation*}

\subsection{A sum over Bessel zeros}\label{subsec:bessel-zero-sum}

The notation and basic facts for Bessel zeros are standard; see
\cite[\S10.21]{DLMF}.  We use the classical Hadamard product
\[
 J_\nu(z)=\frac{(z/2)^\nu}{\Gamma(\nu+1)}
 \prod_{k=1}^{\infty}\left(1-\frac{z^2}{j_{\nu,k}^2}\right),
\]
valid for $\nu>-1$; see, for example, \cite[Eq.~(10.21.15)]{DLMF}.  Its logarithmic
derivative is
\begin{equation}\label{eq:log-derivative}
 \frac{J_\nu'(z)}{J_\nu(z)}
 =\frac{\nu}{z}-2z\sum_{k=1}^{\infty}
 \frac{1}{j_{\nu,k}^2-z^2}.
\end{equation}

The following identity is a special case of Calogero's identity for the zeros
of \(J_\nu\).  We record the short specialization because the precise
normalization is used later.
\begin{lemma}\label{lem:bessel-sum}
One has
\[
 \sum_{k=2}^{\infty}
 \frac{\alpha^2}{j_{\nu,k}^2-\alpha^2}=\frac{N}{4}.
\]
\end{lemma}

\begin{proof}
We use Calogero's identity for the zeros of $J_\nu$:
\[
 \sum_{\substack{k\ge1\\ k\ne m}}
 \frac{1}{j_{\nu,k}^2-j_{\nu,m}^2}
 =
 \frac{\nu+1}{2j_{\nu,m}^2},
 \qquad \nu>-1.
\]
See, for instance, \cite[(1.1)]{BariczJankovMasirevicPoganySzasz2015}.
Taking $m=1$ and using $\alpha=j_{\nu,1}$ gives
\[
 \sum_{k=2}^{\infty}
 \frac{\alpha^2}{j_{\nu,k}^2-\alpha^2}
 =
 \frac{\nu+1}{2}
 =
 \frac{N}{4}.
\qedhere\]
\end{proof}

Now let
\begin{equation*}
 a_k=\frac{\alpha^2}{j_{\nu,k}^2},\qquad
 y_k=\frac{a_k}{1-a_k}\quad(k\ge2).
\end{equation*}
By \eqref{eq:phirvrgr}, \eqref{eq:trhr} and \eqref{eq:log-derivative},
\begin{equation*}
 h(r)
 =-\nu+\alpha r \frac{J_\nu'(\alpha r)}{J_\nu(\alpha r)}
 =-2r^2\sum_{k=1}^{\infty}\frac{a_k}{1-a_kr^2}.
\end{equation*}
Writing $x=\sqrt{1-r^2}$, we can separate the $a_1=1$ term and obtain
\begin{equation}\label{eq:h-x}
 h(r)=-\frac{2(1-x^2)}{x^2}
 -2(1-x^2)\sum_{k=2}^{\infty}\frac{y_k}{1+y_kx^2},
\end{equation}
where, by \cref{lem:bessel-sum},
\begin{equation}\label{eq:sum-y}
 \sum_{k=2}^{\infty}y_k=\frac{N}{4}.
\end{equation}
The remaining $k\ge2$ series in the $x$-variable, and their first derivatives with respect to $x$, converge uniformly on $[0,1]$.  Indeed, $j_{\nu,k}\asymp k$, hence $y_k=O(k^{-2})$, while
\[
\left|\frac{d}{dx}\left(\frac{y_k}{1+y_kx^2}\right)\right|\le 2y_k^2.
\]
All termwise differentiations below are therefore justified.

\subsection{A one-crossing Riccati comparison principle}
\label{subsec:one-sign-comparison}

The Riccati equation will be compared with elementary barriers. For a comparison function \(Q\), the quantity \(E_Q\) below is the \emph{Riccati residual}, namely the left-hand side of \eqref{eq:t-riccati} with \(t\) replaced by \(Q\).  The next lemma is a simple integrating-factor comparison principle: if this residual has one prescribed sign change, then the common endpoint values force \(Q\) to lie on the corresponding side of \(t\).

\begin{lemma}\label{lem:one-sign}
Let \(Q\in C^1([0,1])\) satisfy \(Q(0)=1\) and \(Q(1)=0\), and define its Riccati residual by
\begin{equation}\label{eq:EQ}
 E_Q(r)=rQ'(r)+Q(r)^2+(N-2+2h(r))Q(r)+Cr^2-(N-1).
\end{equation}
Suppose first that there is $r_0\in[0,1]$ such that
$E_Q\ge0$ on $(0,r_0)$ and $E_Q\le0$ on $(r_0,1)$.  Then
$Q\ge t$ on $[0,1]$.  Suppose instead that there is $r_0\in[0,1]$
such that $E_Q\le0$ on $(0,r_0)$ and $E_Q\ge0$ on $(r_0,1)$.  Then
$Q\le t$ on $[0,1]$.
\end{lemma}

\begin{proof}
Let
\[
 F_Q(r)=r\exp\left(\int_0^r\frac{Q(s)-1}{s}\,\dd s\right),
 \qquad 0<r\le1.
\]
Since \(Q\in C^1([0,1])\) and \(Q(0)=1\), the integrand
\((Q(s)-1)/s\) has a continuous extension to \(s=0\), with value \(Q'(0)\).
Thus \(F_Q\) is well-defined and positive on \((0,1]\).  Moreover,
\[
 \frac{F_Q'(r)}{F_Q(r)}
 =
 \frac1r+\frac{Q(r)-1}{r}
 =
 \frac{Q(r)}{r},
\]
and hence \(rF_Q'/F_Q=Q\) on \((0,1)\).  Also
\[
 F_Q(r)\asymp r\quad(r\downarrow0),
 \qquad
 F_Q(r)\asymp1\quad(r\uparrow1).
\]
Put
\[
 w=Q-t,
 \qquad
 \mu_Q=r^{N-2}\phi^2F_Qg.
\]
Subtracting \eqref{eq:t-riccati} from \eqref{eq:EQ} yields
\[
 (\mu_Qw)'=\frac{\mu_Q}{r}E_Q.
\]
Near $0$, we have $F_Q(r)\asymp r$, while
$g(r)\asymp r$ and $\phi(r)\asymp1$; hence $\mu_Q(r)=O(r^N)$.  Near
$1$, we have $F_Q(r)\asymp1$, and
$\phi(r)^2=O((1-r)^2)$.  Since $Q$ and $t$ are bounded and agree at both
endpoints,
\begin{equation*}
 \lim_{r\downarrow0}\mu_Q(r)w(r)
 =\lim_{r\uparrow1}\mu_Q(r)w(r)=0.
\end{equation*}
Moreover $E_Q$ is locally bounded at $0$ and remains bounded at $1$: the
only singular coefficient is $h(r)$, and $h(r)Q(r)=O(1)$ because
$h(r)=O((1-r)^{-1})$ and $Q(r)=O(1-r)$ there.  Thus
$(\mu_Q/r)E_Q\in L^1(0,1)$.
Consequently, with
\[
 I(r)=\int_0^r\frac{\mu_Q(s)}{s}E_Q(s)\dd s,
\]
we have $\mu_Q(r)w(r)=I(r)$ and $I(1)=0$.  In the
positive-to-negative case, $I(r)\ge0$ for $r\le r_0$, while for
$r\ge r_0$ one has
\[
 I(r)=-\int_r^1\frac{\mu_Q(s)}{s}E_Q(s)\dd s\ge0.
\]
Thus $w\ge0$.  The negative-to-positive case is identical with the
inequalities reversed and gives $w\le0$.
\end{proof}

\subsection{Lower and upper Riccati barriers}\label{subsec:riccati-barriers}

We now apply the comparison lemma to two explicit barriers.

\begin{proposition}\label{prop:t-barriers}
For every $0\le r\le1$,
\begin{equation}\label{eq:t-barriers}
 1-r^2\le t(r)\le
 \frac{\kappa_N(1-r^2)}
 {1+(\kappa_N-1)(1-r^2)},
 \qquad
 \kappa_N=\frac{\sqrt{4N+21}}{4}.
\end{equation}
In particular, $0<t(r)<1$ for $0<r<1$, and $g$ is strictly increasing on
$(0,1)$.
\end{proposition}

\begin{proof}
We first record a consequence of the integrating-factor identity used in
\cref{lem:one-sign}.  If $Q\in C^1([0,1])$ satisfies $Q(0)=1$ and
$Q(1)=0$, let $F_Q>0$ solve $rF_Q'/F_Q=Q$ and put $\mu_Q=r^{N-2}\phi^2F_Qg$. By the proof of \cref{lem:one-sign}, we know that
\begin{equation}\label{eq:residual-zero-mean}
 \int_0^1\frac{\mu_Q(r)}{r}E_Q(r)\dd r=0.
\end{equation}

For the lower barrier set $Q_-(r)=x^2=1-r^2$.  The residual
calculation in \cref{lem:barrier-algebra} gives
\begin{equation}\label{eq:lower-residual}
 E_{Q_-}(r)=r^2\Phi_-(x),
\end{equation}
where
\begin{equation}\label{eq:Phi-minus}
 \Phi_-(x)=C-N-5-x^2
 -4\sum_{k=2}^{\infty}\frac{x^2y_k}{1+y_kx^2},
\end{equation}
and
\begin{equation}\label{eq:Phi-minus-derivative}
 \Phi_-'(x)=-2x
 -8x\sum_{k=2}^{\infty}\frac{y_k}{(1+y_kx^2)^2}<0
 \qquad(0<x<1).
\end{equation}
Set $f_-(r)=\Phi_-(\sqrt{1-r^2})$.  Since $x=\sqrt{1-r^2}$ decreases
with $r$ and $\Phi_-$ is strictly decreasing as a function of $x$, the
function $f_-$ is continuous and strictly increasing on $(0,1)$.  Moreover
$E_{Q_-}(r)=r^2f_-(r)$, and \eqref{eq:residual-zero-mean}, applied to
$Q_-$, gives
\[
 \int_0^1\mu_{Q_-}(r)r f_-(r)\dd r=0.
\]
We now spell out the elementary sign argument.  A continuous increasing
function with zero integral against a positive integrable weight is either
identically zero or has the negative-to-positive one-crossing pattern.  Indeed,
if it were strictly positive somewhere and never negative, or strictly negative
somewhere and never positive, the weighted integral would have the corresponding
strict sign.  Otherwise define
\[
 r_-:=\sup\{r\in[0,1]: f_-(r)\le0\}.
\]
Monotonicity gives $f_-\le0$ on $(0,r_-)$ and $f_-\ge0$ on $(r_-,1)$,
with the evident interpretation when the zero set is an interval.  Since
$E_{Q_-}=r^2f_-$, this is exactly the negative-to-positive one-crossing
pattern.  By \cref{lem:one-sign}, $Q_-\le t$.

For the upper barrier, let $\kappa>1$ and set
\begin{equation*}
 D_\kappa(x)=1+(\kappa-1)x^2,
 \qquad
 Q_\kappa(x)=\frac{\kappa x^2}{D_\kappa(x)}.
\end{equation*}
The residual identities in \cref{lem:barrier-algebra} give
\begin{equation}\label{eq:upper-residual}
 E_{Q_\kappa}(r)=r^2\Phi_\kappa(x),
\end{equation}
with
\begin{align}
 \Phi_\kappa(x)
 ={}&C+
 \frac{(-4\kappa^2-\kappa N+4\kappa+N-1)x^2-6\kappa-N+1}{D_\kappa(x)^2}
 \notag\\
 &\quad
 -4Q_\kappa(x)\sum_{i=2}^{\infty}\frac{y_i}{1+y_ix^2},
 \label{eq:Phi-K}
\end{align}
and
\begin{align}
 \Phi_\kappa'(x)
 ={}&\frac{2x(A_0+A_1x^2)}{D_\kappa(x)^3}
 \notag\\
 &+4\sum_{i=2}^{\infty}
 \left[
 Q_\kappa'(x)y_i-
 \left(\frac{Q_\kappa(x)y_i}{1+y_ix^2}\right)'
 \right],
 \label{eq:Phi-K-derivative}
\end{align}
where
\begin{equation}\label{eq:A0-A1}
 A_0=8\kappa^2-10\kappa-N+1,
 \qquad
 A_1=4\kappa^3-8\kappa^2-\kappa N+5\kappa+N-1.
\end{equation}
Every summand in the second line of \eqref{eq:Phi-K-derivative} is
nonnegative, since for $y\ge0$,
\begin{align*}
 Q_\kappa'(x)y-
 \left(\frac{Q_\kappa(x)y}{1+yx^2}\right)'
 ={}&\frac{Q_\kappa'(x)y^2x^2}{1+yx^2}
 +\frac{2xQ_\kappa(x)y^2}{(1+yx^2)^2}\ge0.
\end{align*}
Choose $\kappa=\kappa_N=\sqrt{4N+21}/4$.  Then
\begin{equation*}
 A_0=\frac{(\sqrt{4N+21}-5)^2}{4}\ge0,
 \qquad
 A_0+A_1=\frac{\kappa_N}{4}>0.
\end{equation*}
Because $A_0+A_1x^2$ is affine in $x^2$, it is nonnegative on $[0,1]$.
Hence $\Phi_{\kappa_N}'\ge0$.  Set
$f_+(r)=\Phi_{\kappa_N}(\sqrt{1-r^2})$.  Since $x$ decreases as $r$
increases, $f_+$ is continuous and nonincreasing on $(0,1)$.  Applying
\eqref{eq:residual-zero-mean} to $Q_{\kappa_N}$ and using
$E_{Q_{\kappa_N}}(r)=r^2f_+(r)$ gives
\[
 \int_0^1\mu_{Q_{\kappa_N}}(r)r f_+(r)\dd r=0.
\]
The same elementary sign argument, with monotonicity reversed, gives the positive-to-negative one-crossing pattern.  More explicitly, if $f_+$ is not
identically zero, set
\[
 r_+:=\sup\{r\in[0,1]: f_+(r)\ge0\}.
\]
Then $f_+\ge0$ on $(0,r_+)$ and $f_+\le0$ on $(r_+,1)$, again allowing a
zero interval.  Since $E_{Q_{\kappa_N}}=r^2f_+$, we have
$E_{Q_{\kappa_N}}\ge0$ on $(0,r_+)$ and
$E_{Q_{\kappa_N}}\le0$ on $(r_+,1)$.  If $f_+\equiv0$, the same sign
condition is trivial.  By \cref{lem:one-sign}, $t\le Q_{\kappa_N}$.
This proves \eqref{eq:t-barriers}.
\end{proof}

The barrier bounds imply the following pointwise estimate involving the quotient \(g=v/\phi\) and its derivative.

\begin{proposition}\label{prop:bessel-mismatch}
With $m$ and $d$ as in \eqref{eq:m-d-definitions}, one has
\begin{equation}\label{eq:Bessel-mismatch}
 \frac{r^2\bigl(d(r)-m(r)\bigr)^2}
 {m(r)\bigl(g(1)^2-m(r)\bigr)}\le N,
 \qquad 0<r<1.
\end{equation}
\end{proposition}

\begin{proof}
Put
\begin{equation*}
 s=r^2,
 \qquad z=1-s,
 \qquad Y(s)=\frac{g(\sqrt{s})}{g(1)},
 \qquad T(s)=t(\sqrt{s}).
\end{equation*}
Since \cref{prop:t-barriers} gives $t>0$ on $(0,1)$, the function $g$
is strictly increasing there; hence $0<Y(s)<1$ for $0<s<1$.  Using
$g'/g=t/r$, we obtain
\begin{equation}\label{eq:L-rewrite}
 \frac{r^2(d-m)^2}{m(g(1)^2-m)}
 =\frac{Y(s)^2}{s(1-Y(s)^2)}(z-T(s)^2)^2.
\end{equation}
The lower barrier in \eqref{eq:t-barriers} gives
\[
 \frac{\dd}{\dd s}\log Y(s)^2=\frac{T(s)}{s}\ge\frac{1-s}{s}.
\]
Since $Y(1)=1$, integration from $s$ to $1$ gives
\[
 -\log Y(s)^2
 =\int_s^1\frac{T(\tau)}{\tau}\dd\tau
 \ge\int_s^1\frac{1-\tau}{\tau}\dd\tau
 =-\log s-1+s.
\]
Thus
\begin{equation}\label{eq:Y-upper}
 Y(s)^2\le se^{1-s}=se^z.
\end{equation}

First suppose $2\le N\le10$.  Then $\kappa_N\le2$.  If $x=\sqrt z$,
\[
 1+(\kappa_N-1)x^2-\kappa_Nx
 =(1-x)(1-(\kappa_N-1)x)\ge0,
\]
so the upper barrier in \eqref{eq:t-barriers} is at most $x$.  Hence
\[
 z\le T(s)\le\sqrt z,
 \qquad
 0\le z-T(s)^2\le z(1-z)=zs.
\]
Using \eqref{eq:Y-upper} in \eqref{eq:L-rewrite},
\begin{equation}\label{eq:low-dim-L}
 \frac{Y(s)^2}{s(1-Y(s)^2)}(z-T(s)^2)^2
 \le
 \frac{z^2(1-z)^2e^z}{1-(1-z)e^z}.
\end{equation}
Now
\begin{equation*}
 1-(1-z)e^z=\int_0^z\tau e^\tau\dd\tau\ge\frac{z^2}{2},
 \qquad
 (1-z)^2e^z\le1.
\end{equation*}
Thus the right-hand side of \eqref{eq:low-dim-L} is at most $2\le N$.

Now suppose $N\ge11$, so $\kappa_N>2$.  With $x=\sqrt z$, write
\[
 A(x)=\frac{Q_{\kappa_N}(x)}{x}
 =\frac{\kappa_Nx}{1+(\kappa_N-1)x^2}.
\]
Its squared maximum on $[0,1]$ is
\begin{equation*}
 \max_{0\le x\le1}A(x)^2
 =\frac{\kappa_N^2}{4(\kappa_N-1)}.
\end{equation*}
Since $z\le T(s)\le Q_{\kappa_N}(x)$ and $z=x^2$,
\[
 z\le\frac{T(s)^2}{z}\le A(x)^2\le
 \frac{\kappa_N^2}{4(\kappa_N-1)}.
\]
Consequently
\[
 \left\lvert 1-\frac{T(s)^2}{z}\right\rvert
 \le\max\left\{1-z,
 \frac{\kappa_N^2}{4(\kappa_N-1)}-1\right\}
 \le M_N,
\]
and hence
\begin{equation}\label{eq:z-t-bound}
 \abs{z-T(s)^2}\le M_Nz,
 \qquad
 M_N=\max\left\{1,
 \frac{(\kappa_N-2)^2}{4(\kappa_N-1)}\right\}.
\end{equation}
Moreover, because $\kappa_N\ge4/3$ and
$(\kappa_N-2)^2\le\kappa_N(\kappa_N-1)$,
\begin{equation}\label{eq:M-bound}
 M_N\le\max\left\{1,\frac{\kappa_N}{4}\right\}.
\end{equation}
Combining \eqref{eq:L-rewrite}, \eqref{eq:Y-upper}, and
\eqref{eq:z-t-bound}, we get
\begin{equation*}
 \frac{Y(s)^2}{s(1-Y(s)^2)}(z-T(s)^2)^2
 \le M_N^2\Psi(z),
 \qquad
 \Psi(z)=\frac{z^2e^z}{1-(1-z)e^z}.
\end{equation*}
A direct differentiation gives
\begin{equation*}
 \Psi'(z)=
 \frac{ze^z\bigl[z+2+(z-2)e^z\bigr]}
 {\bigl[1-(1-z)e^z\bigr]^2}\ge0.
\end{equation*}
Indeed, the bracket vanishes at $z=0$ and its derivative is
$1-(1-z)e^z>0$ for $z>0$.  Therefore $\Psi(z)\le\Psi(1)=e$.  Since
$\kappa_N^2=(4N+21)/16$, \eqref{eq:M-bound} gives
\[
 M_N^2\Psi(z)
 \le e\max\left\{1,\frac{4N+21}{256}\right\}<N
 \qquad(N\ge11).
\]
Indeed, if $11\le N\le58$, then the maximum is $1$, so the last bound is
$e<N$.  If $N\ge59$, then the maximum is $(4N+21)/256$, and
$e(4N+21)<3(4N+21)<256N$.
This proves \eqref{eq:Bessel-mismatch} in all dimensions.
\end{proof}

We shall also need the following strict lower bound for \(C\).  It will be
used later in the region \(r\ge1\), where the quotient \(g\) has been extended
as a constant.

\begin{lemma}\label{lem:C-lower-bound}
The constant $C$ in \eqref{eq:nualphabetaC} satisfies
\begin{equation*}
 C>N-1.
\end{equation*}
\end{lemma}

\begin{proof}
Multiply \eqref{eq:g-divergence-ode} by $g$ and integrate over
$[\delta,b]\subset\subset(0,1)$.  After one-dimensional integration by parts,
\begin{align*}
 \left[-r^{N-1}\phi^2g'g\right]_{\delta}^{b}
 +\int_{\delta}^{b}r^{N-1}\phi^2
 \left(g'^2+(N-1)\frac{g^2}{r^2}\right)\dd r =
 C\int_{\delta}^{b}r^{N-1}\phi^2g^2\dd r.
\end{align*}
The boundary term tends to zero as $\delta\downarrow0$ and $b\uparrow1$:
at $0$ it is $O(r^N)$, and at $1$ it is $O((1-r)^3)$ by
\eqref{eq:g-endpoint-asymptotics} and the simple zero of $\phi$.  Passing
to the limit gives
\[
 C\int_0^1r^{N-1}\phi^2g^2\dd r
 =\int_0^1r^{N-1}\phi^2
 \left(g'^2+(N-1)\frac{g^2}{r^2}\right)\dd r.
\]
Since $r^{-2}>1$ on $(0,1)$ and $\phi g\not\equiv0$, the right-hand side
is strictly larger than
$(N-1)\int_0^1r^{N-1}\phi^2g^2\dd r$.
\end{proof}

\section{A weak radial remainder estimate}\label{sec:radial}

In this section we work with the first eigenfunction after extending it by
zero to \(\mathbb R^N\).  We use polar coordinates around a fixed point
\(a\), average \(u^2\) over spheres, and prove a one-dimensional estimate for
these spherical averages.  This estimate will be used in the matrix argument
in \cref{sec:matrix}.  The proof is written in weak form because no regularity
of \(\partial\Omega\) is assumed.

\subsection{Spherical averaging and the radial weak equation}
\label{subsec:radialization}
Here and below, \(\mathcal D'(\mathbb R^N)\) denotes the space of distributions on \(\mathbb R^N\). Let $u\in H^1(\R^N)$ have compact support and suppose
\begin{equation}\label{eq:u2-distribution-general}
 \Delta(u^2)=2\abs{\nabla u}^2-2\alpha^2u^2
 \quad\text{in }\mathcal D'(\R^N).
\end{equation}
The zero extension of a Dirichlet first eigenfunction satisfies this identity; this will be verified in \cref{sec:matrix}.

Fix a center $a\in\R^N$.  Throughout this section, $\dd\sigma$ denotes
the standard surface measure on $\Sph^{N-1}$.  In polar coordinates
$x=a+r\theta$, define for almost every $r>0$
\begin{align*}
 U(r)&=\int_{\Sph^{N-1}}u(a+r\theta)^2\dd\sigma(\theta),
 \\
 E_r(r)&=\int_{\Sph^{N-1}}
 \abs{\partial_ru(a+r\theta)}^2\dd\sigma(\theta),
 \\
 E_\theta(r)&=\int_{\Sph^{N-1}}
 \abs{\nabla_{\Sph^{N-1}}u(a+r\theta)}^2\dd\sigma(\theta).
\end{align*}
Let $w(r)=r^{N-1}$.  The polar decomposition of the $L^2$ and
Dirichlet energies gives, for every $R>0$,
\begin{equation}\label{eq:polar-energy-integrability}
 \int_0^R w(r)U(r)\dd r<\infty,
 \qquad
 \int_0^R w(r)\left(E_r(r)+\frac{E_\theta(r)}{r^2}\right)\dd r<\infty.
\end{equation}

Averaging the distributional identity for \(u^2\) over spheres gives the
following one-dimensional weak equation for \(U\).  We include the proof to
make clear that no boundary regularity is used.

\begin{lemma}\label{lem:radialized-square}
The function $U$ belongs to $W^{1,1}_{\mathrm{loc}}(0,\infty)$,
\begin{equation}\label{eq:U-prime}
 U'(r)=2\int_{\Sph^{N-1}}u\,\partial_ru\dd\sigma
 \quad\text{for a.e. }r\in(0,\infty),
\end{equation}
and
\begin{equation}\label{eq:radial-U-equation}
 (wU')'=2w\left(E_r+\frac{E_\theta}{r^2}-\alpha^2U\right)
 \quad\text{in }\mathcal D'(0,\infty).
\end{equation}
\end{lemma}

\begin{proof}
Because $u^2\in W^{1,1}(\R^N)$, the standard polar slicing theorem gives
$U\in W^{1,1}_{\mathrm{loc}}$ and \eqref{eq:U-prime}.  Let
$\zeta\in C_c^\infty(0,\infty)$.  Then
$x\mapsto\zeta(\abs{x-a})$ is a legitimate test function in
$C_c^\infty(\R^N\setminus\{a\})\subset C_c^\infty(\R^N)$, because its
support stays a positive distance away from $a$.  Testing
\eqref{eq:u2-distribution-general} against this function and using
\[
 \Delta\zeta(r)=\zeta''(r)+\frac{N-1}{r}\zeta'(r),
 \qquad
 \abs{\nabla u}^2=\abs{\partial_ru}^2+\frac{1}{r^2}
 \abs{\nabla_{\Sph^{N-1}}u}^2,
\]
gives \eqref{eq:radial-U-equation} in $\mathcal D'(0,\infty)$ after
one-dimensional integration by parts.
\end{proof}

The next step is to replace \(\sqrt U\) by the strictly positive function
\(\sqrt{U+\varepsilon^2}\).  This avoids division by zero and still gives a
useful lower bound involving \(E_\theta\).

\subsection{The regularized square root}\label{subsec:regularized-wronskian}

For $\varepsilon>0$ set
\begin{equation*}
 z_\varepsilon(r)=\sqrt{U(r)+\varepsilon^2}.
\end{equation*}
We first justify the differentiations involving \(z_\varepsilon\). Let \(I\) be a compact subinterval of \((0,\infty)\). By
\cref{lem:radialized-square}, the function $U$ has an absolutely continuous
representative on $I$, and
\[
 q:=wU'\in W^{1,1}(I),
 \qquad
 q'=2w\left(E_r+\frac{E_\theta}{r^2}-\alpha^2U\right)
 \quad\text{a.e. on }I.
\]
Since $z_\varepsilon\ge\varepsilon$ and $U\in W^{1,1}(I)$, both
$z_\varepsilon$ and $z_\varepsilon^{-1}$ belong to
$W^{1,1}(I)\cap L^\infty(I)$, with
\[
 z_\varepsilon'=\frac{U'}{2z_\varepsilon},
 \qquad
 (z_\varepsilon^{-1})'
 =-\frac{U'}{2z_\varepsilon^3}
 \quad\text{a.e. on }I.
\]
Hence
\[
 wz_\varepsilon'=\frac{q}{2z_\varepsilon}\in W^{1,1}(I),
\]
and differentiating this product gives
\begin{equation}\label{eq:z-chain-rule}
 (wz_\varepsilon')'
 =\frac{(wU')'}{2z_\varepsilon}
  -\frac{wU'^2}{4z_\varepsilon^3}
 \quad\text{in }\mathcal D'(I).
\end{equation}
Since $I$ was arbitrary, \eqref{eq:z-chain-rule} holds in
$\mathcal D'(0,\infty)$.

Define
\begin{equation}\label{eq:eta-definition}
 \eta_\varepsilon
 =(wz_\varepsilon')'+\alpha^2wz_\varepsilon .
\end{equation}
Substituting \eqref{eq:radial-U-equation} into
\eqref{eq:z-chain-rule} gives the locally integrable density
\begin{equation}\label{eq:eta-formula}
 \eta_\varepsilon
 =w\left[
 \frac{E_r+r^{-2}E_\theta}{z_\varepsilon}
 -\frac{U'^2}{4z_\varepsilon^3}
 +\frac{\alpha^2\varepsilon^2}{z_\varepsilon}
 \right].
\end{equation}
By \eqref{eq:U-prime} and Cauchy--Schwarz, we have $U'^2\le4UE_r$. The local integrability follows from
\eqref{eq:polar-energy-integrability}, the bound
$z_\varepsilon\ge\varepsilon$, and the estimate below; in particular,
\[
 \frac{U'^2}{z_\varepsilon^3}
 \le \frac{4UE_r}{z_\varepsilon^3}
 \le \frac{4E_r}{z_\varepsilon}
 \le \frac{4E_r}{\varepsilon}.
\]

Multiplying \eqref{eq:eta-formula} by $z_\varepsilon$ and using
$z_\varepsilon^2=U+\varepsilon^2$ gives
\[
 z_\varepsilon\eta_\varepsilon
 =w\left[
 E_r+\frac{E_\theta}{r^2}
 -\frac{U'^2}{4z_\varepsilon^2}
 +\alpha^2\varepsilon^2
 \right]
 \ge
 w\frac{E_\theta}{r^2}.
\]
Since $z_\varepsilon>0$, this implies
\begin{equation}\label{eq:eta-positive}
 \eta_\varepsilon\ge0,
 \qquad
 z_\varepsilon\eta_\varepsilon
 \ge w\frac{E_\theta}{r^2}.
\end{equation}
Moreover,
\begin{equation}\label{eq:zprime-energy}
 z_\varepsilon'^2
 =\frac{U'^2}{4(U+\varepsilon^2)}
 \le E_r
 \quad\text{for a.e. }r.
\end{equation}
Thus $z_\varepsilon\in W^{1,2}_{\mathrm{loc}}(0,\infty)$.  Also
$wz_\varepsilon\in L^1_{\mathrm{loc}}(0,\infty)$, because
$z_\varepsilon\le U^{1/2}+\varepsilon$ and Cauchy's inequality gives
local integrability of $wU^{1/2}$ from that of $wU$. Consequently,
\begin{equation}\label{eq:wzprime-weak-regularity}
 wz_\varepsilon'\in W^{1,1}_{\mathrm{loc}}(0,\infty),
 \qquad
 (wz_\varepsilon')'=\eta_\varepsilon-\alpha^2wz_\varepsilon
 \quad\text{in }L^1_{\mathrm{loc}}(0,\infty).
\end{equation}

We next compare \(z_\varepsilon\) with the positive radial first
eigenfunction \(\phi\) of the unit ball.  The comparison is encoded in the
weighted Wronskian below.  Its monotonicity will imply that
\(z_\varepsilon/\phi\) is nondecreasing.

\begin{lemma}\label{lem:Wronskian}
Let
\begin{equation*}
 \mathscr W_\varepsilon(r)
 =w(r)\bigl(\phi(r)z_\varepsilon'(r)
            -\phi'(r)z_\varepsilon(r)\bigr),
 \qquad 0<r<1.
\end{equation*}
Then
\begin{equation}\label{eq:Wprime}
 \mathscr W_\varepsilon'=\phi\eta_\varepsilon\ge0,
 \qquad
 \lim_{r\downarrow0}\mathscr W_\varepsilon(r)=0.
\end{equation}
Moreover,
\begin{equation}\label{eq:H-monotone}
 H_\varepsilon(r)=\frac{z_\varepsilon(r)}{\phi(r)}
 \quad\text{is nondecreasing on }(0,1).
\end{equation}
\end{lemma}

\begin{proof}
On compact subintervals of $(0,1)$, \eqref{eq:wzprime-weak-regularity}
gives $\mathscr W_\varepsilon\in W^{1,1}_{\mathrm{loc}}(0,1)$.  Therefore
\eqref{eq:Wprime} follows in the distributional, hence a.e., sense from
$(w\phi')'+\alpha^2w\phi=0$ and \eqref{eq:eta-definition}.  It remains to
justify the value at the origin.

Fix $0<R < 1$. Arguing as above, we obtain
\(wz_\varepsilon'\in W^{1,1}(0,R)\), and therefore
$\ell=\lim_{r\downarrow0}wz_\varepsilon'(r)$ exists.  If $\ell\ne0$, then
for small $r$,
$\abs{z_\varepsilon'(r)}\ge\abs{\ell}/(2w(r))$, which contradicts
\eqref{eq:zprime-energy} because
$\int_0^R w(r)^{-1}\dd r=\infty$ for $N\ge2$.  Therefore
\begin{equation}\label{eq:wzprime-zero}
 \lim_{r\downarrow0}w(r)z_\varepsilon'(r)=0.
\end{equation}

Also $\phi'(r)=O(r)$.  By absolute continuity and the Cauchy--Schwarz inequality with weight \(w\),
for fixed \(R\in(0,1)\) and \(0<r<R\),
\[
 \abs{z_\varepsilon(r)}
 \le \abs{z_\varepsilon(R)}
 +\left(\int_r^Rwz_\varepsilon'^2\dd s\right)^{1/2}
  \left(\int_r^Rw^{-1}\dd s\right)^{1/2}.
\]
Thus
\[
 z_\varepsilon(r)=O(r^{(2-N)/2})\quad(N>2),
 \qquad
 z_\varepsilon(r)=O(\sqrt{\abs{\log r}})\quad(N=2).
\]
Hence $w(r)\phi'(r)z_\varepsilon(r)\to0$.  Together with
\eqref{eq:wzprime-zero}, this proves the second part of
\eqref{eq:Wprime}.  Since
\[
 H_\varepsilon'
 =\frac{\mathscr W_\varepsilon}{w\phi^2},
\]
\eqref{eq:H-monotone} follows.
\end{proof}

We now combine the Wronskian identity with the equation satisfied by \(g\). This gives a one-dimensional remainder estimate for the spherical averages of \(u^2\).

\subsection{The radial remainder estimate}\label{subsec:radial-remainder}

Define
\begin{equation}\label{eq:D-interior}
 \Def_<
 =\int_0^1
 \left(Cg^2-g'^2-(N-1)\frac{g^2}{r^2}\right)U(r)w(r)\dd r.
\end{equation}

The next estimate shows that \(\Def_<\) controls the spherical-gradient term
\(E_\theta\).

\begin{proposition}\label{prop:weak-radial-remainder}
Under \eqref{eq:u2-distribution-general},
\begin{equation}\label{eq:weak-radial-remainder}
 \Def_<\ge
 \int_0^1\frac{g(1)^2-g(r)^2}{r^2}
 E_\theta(r)w(r)\dd r.
\end{equation}
In particular, $\Def_<\ge0$.
\end{proposition}

\begin{proof}
Define $\Def_{<,\varepsilon}$ by replacing $U$ with
$z_\varepsilon^2=U+\varepsilon^2$ in \eqref{eq:D-interior}.  On every
compact interval $[\delta,b]\subset\subset(0,1)$ the functions $g$ and $\phi$ are
smooth and positive, $z_\varepsilon\in H^1(\delta,b)$,
$wz_\varepsilon'\in W^{1,1}(\delta,b)$, and therefore
$H_\varepsilon=z_\varepsilon/\phi\in H^1(\delta,b)$ with
$H_\varepsilon^2\in W^{1,1}(\delta,b)$.  Also
\[
 \mathscr W_\varepsilon=w\phi^2H_\varepsilon'
 \in W^{1,1}(\delta,b),
 \qquad
 \mathscr W_\varepsilon'=\phi\eta_\varepsilon.
\]
Thus the following integration by parts is the standard one-dimensional
Sobolev integration by parts; equivalently, one may first approximate
$H_\varepsilon$ in $H^1(\delta,b)$ by smooth functions.  Multiplying
\eqref{eq:g-divergence-ode} by $gH_\varepsilon^2$ and integrating over
$[\delta,b]$ gives
\begin{equation}\label{eq:truncated-deficit}
 \Def_{<,\varepsilon;\delta,b}
 =\left[-w\phi^2g'gH_\varepsilon^2\right]_{\delta}^{b}
 +2\int_\delta^bgg'H_\varepsilon\mathscr W_\varepsilon\dd r.
\end{equation}
Here $\Def_{<,\varepsilon;\delta,b}$ denotes the corresponding truncated
integral.  The boundary expression equals $-wgg'z_\varepsilon^2$.  At $0$,
\eqref{eq:g-endpoint-asymptotics} gives $wgg'=O(r^N)$.  The growth estimate
for $z_\varepsilon$ in the proof of \cref{lem:Wronskian} gives
$z_\varepsilon(r)=O(r^{(2-N)/2})$ for $N>2$ and
$z_\varepsilon(r)=O(\sqrt{\abs{\log r}})$ for $N=2$.  Thus
$wgg'z_\varepsilon^2=O(r^2)$ for $N>2$ and
$O(r^2\abs{\log r})$ for $N=2$, so the boundary term tends to zero at
$0$.  Near $1$, the estimate \eqref{eq:zprime-energy}, together with
$z_\varepsilon^2=U+\varepsilon^2$ and
\eqref{eq:polar-energy-integrability}, gives
$z_\varepsilon\in H^1(1/2,1)$.  Hence $z_\varepsilon$ is bounded by the
one-dimensional Sobolev embedding, while $g'(r)=O(1-r)$ by
\eqref{eq:g-endpoint-asymptotics}; the boundary term again tends to zero.
Set
\[
 A(r)=Cg(r)^2-g'(r)^2-(N-1)\frac{g(r)^2}{r^2}.
\]
The coefficient $A$ is bounded on $(0,1)$ by the endpoint asymptotics of
$g$, and $z_\varepsilon^2w=(U+\varepsilon^2)w$ belongs to $L^1(0,1)$.
Thus dominated convergence gives
\[
 \Def_{<,\varepsilon;\delta,b}
 \longrightarrow \Def_{<,\varepsilon}
 \qquad(\delta\downarrow0,\ b\uparrow1).
\]
The boundary term has just been shown to tend to zero.  On the other hand,
all factors in the right-hand integrand of \eqref{eq:truncated-deficit} are
nonnegative by \cref{prop:t-barriers,lem:Wronskian}, so monotone
convergence applies to that side.  Passing to the limit in
\eqref{eq:truncated-deficit} gives
\begin{equation*}
 \Def_{<,\varepsilon}
 =2\int_0^1gg'H_\varepsilon\mathscr W_\varepsilon\dd r.
\end{equation*}
Since
\[
 \mathscr W_\varepsilon(r)
 =\int_0^r\phi(t)\eta_\varepsilon(t)\dd t,
\]
Tonelli's theorem and the monotonicity of $H_\varepsilon$ yield
\begin{align*}
 \Def_{<,\varepsilon}
 &=\int_0^1\phi(t)\eta_\varepsilon(t)
   \left(\int_t^12g(r)g'(r)H_\varepsilon(r)\dd r\right)\dd t
 \notag\\
 &\ge\int_0^1\bigl(g(1)^2-g(t)^2\bigr)
   \phi(t)H_\varepsilon(t)\eta_\varepsilon(t)\dd t
 \notag\\
 &=\int_0^1\bigl(g(1)^2-g(t)^2\bigr)
   z_\varepsilon(t)\eta_\varepsilon(t)\dd t
 \notag\\
 &\ge\int_0^1\frac{g(1)^2-g(t)^2}{t^2}
   E_\theta(t)w(t)\dd t,
\end{align*}
where the last step uses \eqref{eq:eta-positive}.  The coefficient
$Cg^2-g'^2-(N-1)g^2/r^2$ is bounded on $(0,1)$, by the endpoint
asymptotics of $g$.  Since $z_\varepsilon^2=U+\varepsilon^2$ and
$\int_0^1w(r)\dd r<\infty$, we have
$\Def_{<,\varepsilon}\to\Def_<$ as $\varepsilon\downarrow0$, proving
\eqref{eq:weak-radial-remainder}.
\end{proof}

It remains to convert the estimate involving \(E_\theta\) into a matrix
estimate.  For this purpose, when \(U(r)>0\), put
\begin{equation*}
 v_r(\theta)=\frac{u(a+r\theta)}{\sqrt{U(r)}},
 \qquad
 \int_{\Sph^{N-1}}v_r^2\dd\sigma=1,
\end{equation*}
and define
\begin{equation}\label{eq:define_B_r}
 \cB(r)=\int_{\Sph^{N-1}}
 \left(\theta\theta^{\top}-\frac{\Id}{N}\right)
 v_r(\theta)^2\dd\sigma(\theta).
\end{equation}
Set $\cB(r)=0$ if $U(r)=0$.

The next lemma shows that the weighted size
\(U(r)\|\cB(r)\|_{\HS}^2\) is controlled by \(E_\theta(r)\).

\begin{lemma}\label{lem:angular-tensor}
For almost every $r>0$,
\begin{equation}\label{eq:angular-tensor-estimate}
 E_\theta(r)\ge\frac{N^2}{4}
 U(r)\norm{\cB(r)}_{\HS}^2.
\end{equation}
\end{lemma}

\begin{proof}
For almost every \(r>0\), the restriction \(\theta\mapsto u(a+r\theta)\) belongs to \(H^1(\Sph^{N-1})\), and \(E_\theta(r)<\infty\). Fix such an $r$. If $U(r)=0$, then the right-hand side of
\eqref{eq:angular-tensor-estimate} is zero and the claim is immediate.
Assume $U(r)>0$, so that $v_r\in H^1(\Sph^{N-1})$ and
$\int_{\Sph^{N-1}}v_r^2\dd\sigma=1$.

Let $Z$ be any trace-free symmetric matrix and set
$Y_Z(\theta)=\theta^{\top}Z\theta$.  Then $Y_Z$ is a degree-two
spherical harmonic,
\[
 -\Delta_{\Sph^{N-1}}Y_Z=2NY_Z,
 \qquad
 \abs{\nabla_{\Sph^{N-1}}Y_Z}\le2\norm{Z}_{\HS}.
\]
Since $Y_Z$ is smooth and $v_r\in H^1$, the following integration by parts
is justified by approximating $v_r$ in $H^1(\Sph^{N-1})$ by smooth
functions:
\begin{align*}
 \ip{\cB(r)}{Z}_{\HS}
 &=\int_{\Sph^{N-1}}Y_Zv_r^2\dd\sigma
  =\frac1{2N}\int_{\Sph^{N-1}}
    \nabla_{\Sph^{N-1}}Y_Z\cdot\nabla_{\Sph^{N-1}}(v_r^2)\dd\sigma \\
 &=\frac1N\int_{\Sph^{N-1}}
    v_r\,\nabla_{\Sph^{N-1}}Y_Z\cdot
    \nabla_{\Sph^{N-1}}v_r\dd\sigma .
\end{align*}
Therefore
\[
 \abs{\ip{\cB(r)}{Z}_{\HS}}
 \le\frac{2}{N}\norm{Z}_{\HS}
 \left(\int_{\Sph^{N-1}}
 \abs{\nabla_{\Sph^{N-1}}v_r}^2\dd\sigma\right)^{1/2}.
\]
Taking the supremum over $\norm{Z}_{\HS}=1$ and using
$E_\theta(r)=U(r)\int_{\Sph^{N-1}}\abs{\nabla_{\Sph^{N-1}}v_r}^2\dd\sigma$
proves \eqref{eq:angular-tensor-estimate}.
\end{proof}

Combining \cref{prop:weak-radial-remainder,lem:angular-tensor}, and writing
\(\dd\rho(r)=w(r)U(r)\,\dd r\), gives
\begin{equation}\label{eq:D-controls-B}
 \Def_<\ge\frac{N^2}{4}
 \int_0^1\frac{g(1)^2-g(r)^2}{r^2}
 \norm{\cB(r)}_{\HS}^2\,\dd\rho(r).
\end{equation}

\section{The matrix trial-space argument}\label{sec:matrix}

We now prove \cref{thm:main}.  Assume the normalization
\eqref{eq:normalization}. Since a domain means a connected open set, the first
Dirichlet eigenvalue is simple.

Let \(u\in H_0^1(\Omega)\) be the nonnegative \emph{normalized first eigenfunction},
so that, weakly in \(\Omega\),
\begin{equation}\label{eq:u-normalized}
 -\Delta u=\alpha^2u,
 \qquad
 \int_\Omega u^2\dd x=1.
\end{equation}
No boundary regularity of $\Omega$ will be used below.
Extend $u$ by zero to $\R^N$.  We verify explicitly that this zero extension
satisfies \eqref{eq:u2-distribution-general}.  For every
$\zeta\in C_c^\infty(\R^N)$, the function $u\zeta$ belongs to
$H_0^1(\Omega)$, and testing \eqref{eq:u-normalized} with $u\zeta$ gives
\[
 \int_\Omega \zeta\abs{\nabla u}^2\dd x
 +\int_\Omega u\nabla u\cdot\nabla\zeta\dd x
 =\alpha^2\int_\Omega u^2\zeta\dd x.
\]
With the zero extension understood, $\nabla(u^2)=2u\nabla u$ in
$L^1(\R^N)$.  Hence
\[
 \langle \Delta(u^2),\zeta\rangle
 =-\int_{\R^N}\nabla(u^2)\cdot\nabla\zeta\dd x
 =2\int_{\R^N}\zeta\abs{\nabla u}^2\dd x
  -2\alpha^2\int_{\R^N}u^2\zeta\dd x.
\]
Therefore
\begin{equation*}
 \Delta(u^2)=2\abs{\nabla u}^2-2\alpha^2u^2
 \quad\text{in }\mathcal D'(\R^N).
\end{equation*}
Thus all results of \cref{sec:radial} apply.

\subsection{The Weinberger center}\label{subsec:weinberger-center}
We next choose the center \(a\) so that the components of the vector-valued
trial function defined below have zero \(u^2\)-mean. In this subsection, \(g\) denotes the constant extension
of the quotient \(v/\phi\):
\[
 g(r)=\frac{v(r)}{\phi(r)}\quad 0<r<1,
 \qquad
 g(r)=g(1)\quad r\ge1.
\]
Since \(v\) and \(\phi\) have simple zeros at \(r=1\), the quotient has a
finite positive limit \(g(1)>0\).  Moreover, by
\eqref{eq:g-endpoint-asymptotics}, \(g'(r)\to0\) as \(r\uparrow1\).  Hence
the constant extension of \(g\) to \([1,\infty)\) is Lipschitz.

The choice of \(a\) is the weighted center construction used in Weinberger's
method; see \cite{Weinberger1956,Laugesen2021}.  For completeness we give the
short existence proof.  Let
\[
 \Gamma(r)=\int_0^r g(s)\,\dd s .
\]
Since \(g\) is bounded, \(\Gamma\) is globally Lipschitz.  Since
\(g(r)=g(1)>0\) for \(r\ge1\), \(\Gamma\) grows linearly at infinity.  Define
\[
 \mathcal F(a)=\int_\Omega\Gamma(\abs{x-a})u(x)^2\,\dd x.
\]
Then \(\mathcal F\) is continuous, in fact Lipschitz.  It is also coercive:
if \(R_\Omega=\sup_{x\in\Omega}\abs{x}<\infty\) and
\(\abs{a}\ge R_\Omega+1\), then \(\abs{x-a}\ge1\) for \(x\in\Omega\), and
\[
 \mathcal F(a)
 \ge \Gamma(1)+g(1)(\abs{a}-R_\Omega-1).
\]
Here we used \(\int_\Omega u^2\,\dd x=1\).  Hence
\(\mathcal F(a)\to+\infty\) as \(\abs a\to\infty\), and \(\mathcal F\) has a
minimizer \(a\).

The map \(y\mapsto\Gamma(\abs y)\) is \(C^1\) at \(y=0\), with gradient zero,
because \(g(r)=O(r)\) as \(r\downarrow0\).  For \(x\ne a\),
\[
 \nabla_a\Gamma(\abs{x-a})
 =
 -g(r)\theta,
 \qquad
 r=\abs{x-a},
 \quad
 \theta=\frac{x-a}{\abs{x-a}},
\]
and the value at \(x=a\) is interpreted as zero.  Since
\(\abs{g(r)\theta}\le g(1)\), dominated convergence justifies differentiation
under the integral sign.  The first-order condition at the minimizer gives
\begin{equation}\label{eq:center-condition}
 \int_\Omega g(r)\theta\,u^2\,\dd x=0,
 \qquad
 r=\abs{x-a},
 \quad
 \theta=\frac{x-a}{\abs{x-a}},
\end{equation}
again with the integrand interpreted as zero at \(x=a\).

Define the vector-valued trial map
\begin{equation}\label{eq:F-vector}
 F(x)=g(r)\theta\in\R^N\quad (x\ne a),
 \qquad F(a)=0.
\end{equation}
This pointwise definition at \(a\) is consistent with continuity, because
\(g(r)=O(r)\) as \(r\downarrow0\).  The map \(F\) is globally Lipschitz.
Indeed, for \(0<r<1\),
\[
 DF=g'(r)P+\frac{g(r)}{r}(\Id-P),
 \qquad P=\theta\theta^{\top}.
\]
The endpoint asymptotics of \(g\) give boundedness of \(g'\) and \(g/r\)
near \(0\).  At \(r=1\), the constant extension has radial derivative \(0\),
matching \(g'(1)=0\), while the angular coefficient \(g(r)/r\) remains
bounded.  For \(r>1\),
\[
 DF=\frac{g(1)}{r}(\Id-P),
\]
which is bounded as well.  Thus \(F\) is continuous and has uniformly bounded
a.e.\ derivative, and hence is globally Lipschitz.  Its components have zero
\(u^2\)-mean by \eqref{eq:center-condition}.

\subsection{Trace Rayleigh--Ritz reduction}\label{subsec:trace-ritz-reduction}

We use the following matrix form of the Rayleigh--Ritz principle.  It converts
an \(N\)-dimensional family of functions, orthogonal to the first eigenfunction,
into a lower bound for the sum of the reciprocal spectral gaps.  This is the
trace form of Hersch's variational principle; see
\cite{Hersch1961,HileXu1993}. For a recent use of the same matrix viewpoint in the Neumann
Ashbaugh--Benguria conjecture, see \cite{HeLiTang2026}.

\begin{lemma}\label{lem:trace-Ritz}
Let $u$ be a normalized simple first Dirichlet eigenfunction and let
$f_1,\dots,f_N\in W^{1,\infty}(\Omega)$ satisfy
$\int_\Omega f_i u^2=0$.  Define
\begin{equation*}
 M_{ij}=\int_\Omega f_if_ju^2\dd x,
 \qquad
 K_{ij}=\int_\Omega\nabla f_i\cdot\nabla f_j\,u^2\dd x.
\end{equation*}
If both $M$ and $K$ are positive definite, then
\begin{equation*}
 \sum_{i=1}^{N}\frac{1}{\lambda_{i+1}-\lambda_1}
 \ge\tr(K^{-1}M).
\end{equation*}
\end{lemma}

\begin{proof}
Set $\psi_i=uf_i$ for $i=1,\ldots,N$, and let $V=\operatorname{span}\{\psi_1,\ldots,\psi_N\}$. Since multiplication by a \(W^{1,\infty}\) function preserves
\(H_0^1(\Omega)\), we have \(\psi_i\in H_0^1(\Omega)\).  Moreover,
\[
 \int_\Omega \psi_i u\,\dd x
 =
 \int_\Omega f_i u^2\,\dd x
 =0,
\]
and hence \(V\subset u^\perp\), where the orthogonality is in
\(L^2(\Omega)\).  Since \(M\) is positive definite, the functions
\(\psi_1,\ldots,\psi_N\) are linearly independent in \(L^2(\Omega)\), so
\(\dim V=N\).

For \(v\in H_0^1(\Omega)\), put
\[
 \mathcal Q(v)
 =
 \int_\Omega\bigl(|\nabla v|^2-\lambda_1 v^2\bigr)\dd x.
\]
Since \(\lambda_1\) is simple, the first eigenspace is
\(\operatorname{span}\{u\}\).  Hence the shifted Dirichlet form $\mathcal Q(v)$ restricted to \(H_0^1(\Omega)\cap u^\perp\) has eigenvalues
\[
 \lambda_2-\lambda_1,\lambda_3-\lambda_1,\ldots,
\]
counted with multiplicity.  Therefore the min--max principle gives, for
\(k=1,\ldots,N\),
\[
 \lambda_{k+1}-\lambda_1
 =
 \inf_{\substack{L\subset H_0^1(\Omega)\cap u^\perp\\ \dim L=k}}
 \ \sup_{0\ne v\in L}
 \frac{\mathcal Q(v)}{\int_\Omega v^2\dd x}.
\]
Now define the corresponding restricted min--max quantities on \(V\) by
\[
 \eta_k
 =
 \inf_{\substack{L\subset V\\ \dim L=k}}
 \ \sup_{0\ne v\in L}
 \frac{\mathcal Q(v)}{\int_\Omega v^2\dd x},
 \qquad k=1,\ldots,N.
\]
Because \(V\subset u^\perp\), the infimum defining
\(\lambda_{k+1}-\lambda_1\) is taken over a larger class of subspaces than
the infimum defining \(\eta_k\).  Therefore
\[
 \lambda_{k+1}-\lambda_1\le \eta_k,
 \qquad k=1,\ldots,N.
\]
Since \(K\) is positive definite, each \(\eta_k\) is positive.  Hence
\[
 \sum_{k=1}^{N}
 \frac{1}{\lambda_{k+1}-\lambda_1}
 \ge
 \sum_{k=1}^{N}\frac{1}{\eta_k}.
\]

It remains to compute the last sum in terms of \(M\) and \(K\). For
\(v\in V\), write
\[
 v=\sum_{i=1}^{N}a_i\psi_i=uf,
 \qquad
 f=\sum_{i=1}^{N}a_i f_i,
\]
where \(a=(a_1,\ldots,a_N)^\top\in\mathbb R^N\).  Then
\[
 \int_\Omega v^2\dd x=a^\top M a.
\]

We next compute the shifted energy. Testing the first-eigenfunction equation \(-\Delta u=\lambda_1u\) with \(uf^2\) gives
\[
 \int_\Omega
 \left(f^2|\nabla u|^2+2uf\nabla u\cdot\nabla f\right)\dd x
 =
 \lambda_1\int_\Omega u^2f^2\dd x.
\]
Using $|\nabla(uf)|^2 = f^2|\nabla u|^2+2uf\nabla u\cdot\nabla f+u^2|\nabla f|^2$, we obtain
\[
 \mathcal Q(uf)
 =
 \int_\Omega\left(|\nabla(uf)|^2-\lambda_1u^2f^2\right)\dd x
 =
 \int_\Omega u^2|\nabla f|^2\dd x.
\]
Consequently, for \(f=\sum_{i=1}^N a_i f_i\) we have
\[
 \mathcal Q(v)=\mathcal Q(uf)
 =
 \sum_{i,j=1}^N a_i a_j
 \int_\Omega u^2\nabla f_i\cdot\nabla f_j\,\dd x
 =
 a^\top K a.
\]
Thus, on \(V\), the shifted Rayleigh quotient is
\[
 \frac{\mathcal Q(v)}{\int_\Omega v^2\dd x}
 =
 \frac{a^\top K a}{a^\top M a}.
\]

Let $A=K^{-1/2}MK^{-1/2}$. This is a symmetric positive definite matrix. If \(b=K^{1/2}a\), then
\[
 \frac{a^\top K a}{a^\top M a}
 =
 \frac{|b|^2}{b^\top A b}.
\]
Let \(\sigma_1\ge\cdots\ge\sigma_N>0\) be the eigenvalues of \(A\).  By the
finite-dimensional Courant--Fischer principle applied to \(A\),
\[
 \frac1{\eta_k}=\sigma_k,\qquad k=1,\ldots,N.
\]
Consequently,
\[
 \sum_{k=1}^{N}\frac1{\eta_k}
 =
 \sum_{k=1}^{N}\sigma_k
 =
 \tr(A)
 =
 \tr(K^{-1/2}MK^{-1/2})
 =
 \tr(K^{-1}M),
\]
where the last equality follows from cyclicity of the trace.  Combining this
with the previous inequality proves
\[
 \sum_{k=1}^{N}
 \frac{1}{\lambda_{k+1}-\lambda_1}
 \ge
 \tr(K^{-1}M).
\qedhere\]
\end{proof}

Apply \cref{lem:trace-Ritz} with \(f_i=F_i\) for \(i=1,\ldots,N\), where $F$ is defined in \eqref{eq:F-vector}. Thus
\[
 M=\int_\Omega FF^{\top}u^2\,\dd x,
 \qquad
 K=\int_\Omega (DF)(DF)^{\top}u^2\,\dd x,
\]
that is,
\[
 M_{ij}=\int_\Omega F_iF_j u^2\,\dd x,
 \qquad
 K_{ij}=\int_\Omega \nabla F_i\cdot\nabla F_j\,u^2\,\dd x.
\]
Here \(DF\) denotes the a.e. Jacobian matrix of the Lipschitz map \(F\),
with entries \((DF)_{ij}=\partial_jF_i\).  We first verify that \(M\) and
\(K\) are positive definite.

For $\xi\ne0$,
\[
 \xi^{\top}M\xi
 =\int_\Omega g(r)^2(\xi\cdot\theta)^2u^2\dd x
 =\int_\Omega\left(\frac{g(r)}{r}\right)^2
   \bigl(\xi\cdot(x-a)\bigr)^2u^2\dd x.
\]
Since $g(r)/r>0$ for $r>0$,
equality $\xi^{\top}M\xi=0$ would imply that $u=0$ a.e. outside the
affine hyperplane $a+\xi^\perp$.  This hyperplane has zero
$N$-dimensional Lebesgue measure, and $u^2\dd x$ is absolutely continuous
with respect to Lebesgue measure.  Hence $u=0$ a.e. in $\Omega$, contradicting
$\int_\Omega u^2\dd x=1$.  Therefore $M$ is positive definite.

By the variational characterization of \(\lambda_2\), and since
\(\lambda_1\) is simple, the shifted Dirichlet form
\[
 \mathcal Q(v)=\int_\Omega\bigl(\abs{\nabla v}^2-\lambda_1v^2\bigr)\,\dd x
\]
satisfies
\[
 \mathcal Q(v)\ge(\lambda_2-\lambda_1)\int_\Omega v^2\,\dd x
 \qquad
 \text{for every }v\in H_0^1(\Omega)\cap u^\perp .
\]
We now prove that \(K\) is positive definite.  Suppose that
\(\xi^{\top}K\xi=0\), and set $v=u\,\xi\cdot F$. Since \(F\in W^{1,\infty}(\Omega;\mathbb R^N)\), we have
\(v\in H_0^1(\Omega)\).  Moreover, by \eqref{eq:center-condition},
\[
 \int_\Omega vu\,\dd x
 =
 \int_\Omega (\xi\cdot F)u^2\,\dd x
 =
 \xi\cdot\int_\Omega F u^2\,\dd x
 =0.
\]
Thus \(v\in u^\perp\).  Applying the identity
\[
 \mathcal Q(uf)=\int_\Omega u^2\abs{\nabla f}^2\,\dd x
\]
with \(f=\xi\cdot F\), we obtain
\[
 \mathcal Q(v)
 =
 \int_\Omega u^2\abs{\nabla(\xi\cdot F)}^2\,\dd x
 =
 \xi^{\top}K\xi
 =0.
\]
The spectral-gap estimate above therefore gives \(v=0\).  Hence
\[
 \xi^{\top}M\xi
 =
 \int_\Omega (\xi\cdot F)^2u^2\,\dd x
 =
 \int_\Omega v^2\,\dd x
 =0.
\]
Since \(M\) is positive definite, this forces \(\xi=0\).  Thus \(K\) is
positive definite as well.  Consequently,
\begin{equation}\label{eq:Ritz-main}
 \sum_{i=1}^{N}\frac{1}{\lambda_{i+1}-\lambda_1}
 \ge\tr(K^{-1}M).
\end{equation}

Writing $P=\theta\theta^{\top}$, direct differentiation of
\eqref{eq:F-vector} gives
\begin{align*}
 M&=\int_\Omega m(r)P\,u^2\dd x,
 \\
 K&=\int_\Omega
 \left(g'(r)^2P+\frac{g(r)^2}{r^2}(\Id-P)\right)u^2\dd x.
\end{align*}
Here $m$ and $d$ are the functions defined in \eqref{eq:m-d-definitions};
for $r\ge1$, $g'=0$, $m=g(1)^2$, and $d(r)=g(1)^2/r^2$.  Set
\begin{equation*}
 S=\int_\Omega d(r)P\,u^2\dd x.
\end{equation*}
Then
\begin{equation*}
 K=b\Id-S,
 \qquad
 b=\int_\Omega\frac{g(r)^2}{r^2}u^2\dd x.
\end{equation*}

Let
\begin{equation}\label{eq:def:G_M0_S0}
 G=\tr M=\int_\Omega m(r)u^2\dd x,
 \qquad
 M_0=M-\frac{G}{N}\Id,
 \qquad
 S_0=S-\frac{\tr S}{N}\Id.
\end{equation}
Since $\tr K=Nb-\tr S$, we have
\begin{equation*}
 K=\frac{\tr K}{N}\Id-S_0.
\end{equation*}
Define the trace-free mismatch matrix
\begin{equation*}
 \cR=S_0-M_0
 =\int_\Omega(d(r)-m(r))
 \left(P-\frac{\Id}{N}\right)u^2\dd x.
\end{equation*}
This matrix measures the mismatch between the trace-free parts of \(S\) and \(M\). Then
\begin{align}
 \tr(KM)
 &=\frac{G}{N}\tr K-\ip{S_0}{M_0}_{\HS},
 \label{eq:trKM-1}\\
 -\ip{S_0}{M_0}_{\HS}
 &=-\norm{M_0+\tfrac12\cR}_{\HS}^2
   +\frac14\norm{\cR}_{\HS}^2
 \le\frac14\norm{\cR}_{\HS}^2.
 \label{eq:square-completion}
\end{align}

Set
\begin{equation}\label{eq:def:DefCG-trK}
\Def=CG-\tr K.
\end{equation}
This scalar deficit will be split into its interior and exterior parts below.
The central estimate is
\begin{equation}\label{eq:mismatch-target}
 \norm{\cR}_{\HS}^2\le\frac{4G\Def}{N}.
\end{equation}
We prove it next.

\subsection{Interior mismatch}\label{subsec:interior-mismatch}

Use the polar notation of \cref{sec:radial} around the Weinberger center
\(a\), and set
\[
 \dd\rho(r)=r^{N-1}U(r)\dd r.
\]
By the definition of \(\cB(r)\) in \eqref{eq:define_B_r}, the inner spherical integral in the formula
for \(\cR\) is \(U(r)\cB(r)\). Hence
\begin{equation}\label{eq:def:cR}
 \cR=\int_0^\infty(d(r)-m(r))\cB(r)\,\dd\rho(r).
\end{equation}

Taking traces in the formula for \(K\) and using
\(\tr P=1\), \(\tr(\Id-P)=N-1\), we have
\[
 \tr K
 =
 \int_\Omega
 \left(g'(r)^2+(N-1)\frac{g(r)^2}{r^2}\right)u^2\,\dd x.
\]
Since \(G=\tr M=\int_\Omega m(r)u^2\,\dd x\), it follows that
\[
 \Def=CG-\tr K
 =
 \int_\Omega
 \left(Cm(r)-g'(r)^2-(N-1)\frac{g(r)^2}{r^2}\right)u^2\,\dd x.
\]
Equivalently, in the polar notation
\(\dd\rho(r)=r^{N-1}U(r)\,\dd r\),
\[
 \Def
 =
 \int_0^\infty
 \left(Cm-g'^2-(N-1)\frac{g^2}{r^2}\right)\dd\rho.
\]

We split the relevant quantities in \eqref{eq:def:G_M0_S0}, \eqref{eq:def:DefCG-trK} and \eqref{eq:def:cR} at \(r=1\):
\[
 \cR=\cR_{<}+\cR_{>},
 \qquad
 G=G_{<}+G_{>},
 \qquad
 \Def=\Def_{<}+\Def_{>},
\]
where
\[
 \cR_{<}=\int_0^1(d-m)\cB\,\dd\rho,
 \qquad
 \cR_{>}=\int_1^\infty(d-m)\cB\,\dd\rho,
\]
and
\begin{align*}
 G_{<}&=\int_0^1m\,\dd\rho,
 &
 \Def_{<}&=\int_0^1
 \left(Cm-g'^2-(N-1)\frac{g^2}{r^2}\right)\dd\rho,\\
 G_{>}&=\int_1^\infty m\,\dd\rho,
 &
 \Def_{>}&=\int_1^\infty
 \left(Cm-g'^2-(N-1)\frac{g^2}{r^2}\right)\dd\rho.
\end{align*}
By \cref{prop:weak-radial-remainder}, \(\Def_{<}\ge0\). Since \(g\) is
strictly increasing on \((0,1)\), one has \(0<m(r)<g(1)^2\) for
\(0<r<1\). Therefore the following weighted Cauchy--Schwarz estimate is
legitimate:
\begin{align}
 \norm{\cR_{<}}_{\HS}^2
 &\le
 \left[\int_0^1\frac{r^2(d-m)^2}{g(1)^2-m}\,\dd\rho\right]
 \left[\int_0^1\frac{g(1)^2-m}{r^2}
       \norm{\cB}_{\HS}^2\,\dd\rho\right]
 \notag\\
 &\le
 \left[N\int_0^1m\,\dd\rho\right]
 \left[\frac{4}{N^2}\Def_{<}\right]
 =\frac{4}{N}G_{<}\Def_{<}.
 \label{eq:interior-mismatch}
\end{align}
Here the first factor is bounded by \cref{prop:bessel-mismatch}, while the
second factor is bounded by \eqref{eq:D-controls-B}.

\subsection{Exterior mismatch}\label{subsec:exterior-mismatch}

For \(r\ge1\), the constant extension gives \(g(r)=g(1)\) and
\(g'(r)=0\). Therefore
\[
 d(r)-m(r)=g(1)^2\left(\frac{1}{r^2}-1\right),
\]
and
\[
 \Def_{>}=g(1)^2\int_1^\infty
 \left(C-\frac{N-1}{r^2}\right)\dd\rho(r)\ge0
\]
by \cref{lem:C-lower-bound}.

For \(U(r)>0\), set
\[
\Pi(r)=\int_{\Sph^{N-1}}\theta\theta^{\top}v_r^2\,\dd\sigma,
\]
and, for \(U(r)=0\), set \(\Pi(r)=\Id/N\). Then \(\Pi(r)\) is positive
semidefinite with trace one, and \(\cB(r)=\Pi(r)-\Id/N\) in both cases.
Hence
\begin{equation}\label{eq:B-universal-bound}
 \norm{\cB(r)}_{\HS}^2
 =\tr(\Pi(r)^2)-\frac1N\le\frac{N-1}{N}.
\end{equation}

For \(\tau\in[0,1]\), \cref{lem:C-lower-bound} implies
\begin{equation}\label{eq:outside-scalar-ratio}
 \frac{(1-\tau)^2}{C-(N-1)\tau}\le\frac1C.
\end{equation}
Using the Hilbert--Schmidt Cauchy--Schwarz inequality, followed by
\eqref{eq:B-universal-bound} and \eqref{eq:outside-scalar-ratio} with
\(\tau=r^{-2}\), we obtain
\begin{align}
 \norm{\cR_{>}}_{\HS}^2
 &\le G_{>}\int_1^\infty
 g(1)^2\left(1-\frac1{r^2}\right)^2
 \norm{\cB(r)}_{\HS}^2\,\dd\rho(r)
 \notag\\
 &\le \frac{N-1}{NC}G_{>}\Def_{>}
 \le \frac4N G_{>}\Def_{>}.
 \label{eq:exterior-mismatch}
\end{align}
In the last step we used the weaker bound \((N-1)/(NC)<1/N< 4/N\), which follows from \cref{lem:C-lower-bound}.

Combining \eqref{eq:interior-mismatch} and \eqref{eq:exterior-mismatch},
we get
\begin{align*}
 \norm{\cR}_{\HS}
 &\le\frac{2}{\sqrt N}
 \left(\sqrt{G_{<}\Def_{<}}+\sqrt{G_{>}\Def_{>}}\right)\\
 &\le\frac{2}{\sqrt N}
 \sqrt{(G_{<}+G_{>})(\Def_{<}+\Def_{>})}
 =\frac{2}{\sqrt N}\sqrt{G\Def}.
\end{align*}
Squaring this inequality proves \eqref{eq:mismatch-target}.

\subsection{Completion of the proof}\label{subsec:completion}

Combining \eqref{eq:trKM-1}, \eqref{eq:square-completion}, and
\eqref{eq:mismatch-target}, and using \(\Def=CG-\tr K\), we obtain
\begin{equation}\label{eq:trKM-final}
 \tr(KM)
 \le\frac{G}{N}\tr K+\frac{G\Def}{N}
=\frac{G}{N}(CG-\Def)+\frac{G\Def}{N}
 =\frac{CG^2}{N}.
\end{equation}

For positive definite \(K\) and \(M\), the Hilbert--Schmidt
Cauchy--Schwarz inequality gives
\begin{equation}\label{eq:matrix-CS}
 \tr(K^{-1}M)\tr(KM)\ge(\tr M)^2=G^2.
\end{equation}
Since \(\tr(KM)>0\), \eqref{eq:trKM-final} and \eqref{eq:matrix-CS}
imply
\[
 \tr(K^{-1}M)
 \ge
 \frac{G^2}{\tr(KM)}
 \ge
 \frac{G^2}{CG^2/N}
 =
 \frac{N}{C}.
\]
Together with \eqref{eq:Ritz-main}, this proves
\eqref{eq:normalized-goal}.

It remains to undo the normalization. Let \(\Omega_0\) be an arbitrary
bounded domain, and set
\[
 s=\frac{\sqrt{\lambda_1(\Omega_0)}}{\alpha},
 \qquad
 \widetilde\Omega=s\Omega_0.
\]
Then the scaling law for Dirichlet eigenvalues gives
\[
 \lambda_k(\widetilde\Omega)=s^{-2}\lambda_k(\Omega_0)
 =
 \frac{\alpha^2}{\lambda_1(\Omega_0)}\lambda_k(\Omega_0),
 \qquad
 \lambda_1(\widetilde\Omega)=\alpha^2.
\]
Applying \eqref{eq:normalized-goal} to \(\widetilde\Omega\) gives
\[
 \frac{\lambda_1(\Omega_0)}{\alpha^2}
 \sum_{i=1}^N
 \frac{1}{\lambda_{i+1}(\Omega_0)-\lambda_1(\Omega_0)}
 \ge
 \frac{N}{\beta^2-\alpha^2}.
\]
Equivalently,
\[
 \sum_{i=1}^N
 \frac{\lambda_1(\Omega_0)}
 {\lambda_{i+1}(\Omega_0)-\lambda_1(\Omega_0)}
 \ge
 \frac{N\alpha^2}{\beta^2-\alpha^2}
 =
 \frac{N}{\beta^2/\alpha^2-1}.
\]
Since \(\alpha=j_{\nu,1}\) and \(\beta=j_{\nu+1,1}\), this is
\eqref{eq:main}.

It remains to discuss equality.  First suppose that
$\SobCap(\Omega\triangle B_R(a))=0$ for some ball $B_R(a)$.  By the
capacitary characterization of Dirichlet spaces recalled in \cref{subsec:main-results},
\[
 H_0^1(\Omega)=H_0^1(B_R(a)).
\]
Thus the two Dirichlet forms have the same form domain and the same energy
integral, and hence the associated Dirichlet spectra coincide; see
\cite[Ch.~4]{EvansGariepy2015} and
\cite[Ch.~2]{BucurButtazzo2005}.  Separation of variables gives the ball levels
$j_{\nu+\ell,k}^2/R^2$.  The standard zero ordering
$j_{\nu+1,1}<j_{\nu,2}$, together with the monotonicity of
$\mu\mapsto j_{\mu,1}$, gives
$j_{\nu+1,1}<j_{\nu+\ell,1}$ for every $\ell\ge2$.  Hence the first level
after $\lambda_1$ is precisely the degree-one level, with multiplicity $N$:
\[
 \lambda_1=\frac{\alpha^2}{R^2},
 \qquad
 \lambda_2=\cdots=\lambda_{N+1}=\frac{\beta^2}{R^2}.
\]
Thus equality holds.

Conversely, assume that equality holds in \eqref{eq:main}. Since both sides
of \eqref{eq:main} are invariant under dilations, we may scale \(\Omega\) so
that $\lambda_1(\Omega)=\alpha^2$. In the notation of the proof above,
\[
 \frac{N}{C}
 =
 \sum_{i=1}^{N}\frac{1}{\lambda_{i+1}-\lambda_1}
 \ge \tr(K^{-1}M)
 \ge \frac{G^2}{\tr(KM)}
 \ge \frac{N}{C}.
\]
Hence equality holds at every step.  Equality in the Hilbert--Schmidt
Cauchy--Schwarz inequality \eqref{eq:matrix-CS}, applied to
\(K^{-1/2}M^{1/2}\) and \(K^{1/2}M^{1/2}\), gives
\[
 K^{-1/2}M^{1/2}=cK^{1/2}M^{1/2}
\]
for some \(c>0\).  Indeed, the Hilbert--Schmidt inner product of these two matrices is \(\tr M=G>0\), so the proportionality constant is positive. Since \(M\) is positive definite, we may multiply on the right by \(M^{-1/2}\) and obtain $K^{-1/2}=cK^{1/2}$. Multiplying on the left by \(K^{1/2}\) gives \(\Id=cK\).  Hence \(K\) is a
scalar matrix.

Since equality also holds in \eqref{eq:trKM-final}, we have $\tr(KM)=CG^2/N$. Writing \(K=k\Id\) and using \(\tr M=G\), we get $kG=CG^2/N$. Thus $K=(CG/N)\Id$ and therefore
\[
 \Def=CG-\tr K=0.
\]
By \cref{prop:weak-radial-remainder} and the positivity of the exterior integrand established in \cref{subsec:exterior-mismatch}, both \(\Def_{<}\) and \(\Def_{>}\) are
nonnegative.  Since \(\Def=\Def_{<}+\Def_{>}=0\), we have
\[
 \Def_{<}=\Def_{>}=0.
\]
In particular,
\[
 0=\Def_{>}
 =g(1)^2\int_1^\infty
 \left(C-\frac{N-1}{r^2}\right)\dd\rho(r).
\]
Since
\[
 C-\frac{N-1}{r^2}\ge C-(N-1)>0
 \qquad (r\ge1)
\]
by \cref{lem:C-lower-bound}, the last identity implies 
\[
 0=\rho((1,\infty))
 =
 \int_1^\infty r^{N-1}U(r)\,\dd r
 =
 \int_{\R^N\setminus B_1(a)}u^2\,\dd x.
\]
Thus the zero extension of \(u\) vanishes a.e. outside the unit ball
\[
 B:=B_1(a).
\]

Since \(B\) is smooth, the zero-extension characterization of \(H_0^1(B)\)
applies; see \cite[Proposition~9.18]{Brezis2011}.  Since \(u\in H^1(\R^N)\)
vanishes a.e. on \(\R^N\setminus B\), it follows that \(u|_B\in H_0^1(B)\). Since \(u\) is a first eigenfunction with eigenvalue \(\alpha^2\), testing the
weak equation with \(u\) gives
\[
 \frac{\int_B|\nabla u|^2\dd x}{\int_Bu^2\dd x}
 =
 \frac{\int_{\R^N}|\nabla u|^2\dd x}{\int_{\R^N}u^2\dd x}
 =
 \alpha^2
 =
 \lambda_1(B).
\]
Thus \(u|_B\) attains the Rayleigh minimum for the unit ball, and hence is a
first Dirichlet eigenfunction of \(B\).  Therefore, for some constant
\(c_0>0\),
\[
 u(x)=c_0\phi(|x-a|)\quad\text{in }B,
 \qquad
 u=0\quad\text{a.e. in }\R^N\setminus B,
\]
where we recall that \(\phi(r)=r^{-\nu}J_\nu(\alpha r)\) is the positive radial first
eigenfunction of the unit ball.

Let \(\tilde u\) be the quasi-continuous representative of \(u\).  Since \(u\in H_0^1(\Omega)\), the capacitary characterization of \(H_0^1(\Omega)\) in \eqref{eq:cap-def-of-h01} gives
\[
 \tilde u=0
 \quad\text{quasi-everywhere on }\R^N\setminus\Omega.
\]
On the other hand, \(u\) agrees a.e. in \(B\) with the continuous function
\(c_0\phi(|x-a|)\).  Since this continuous function is quasi-continuous and
represents the same \(H^1(B)\)-class as \(u|_B\), the uniqueness of
quasi-continuous representatives gives
\[
 \tilde u=c_0\phi(|x-a|)
 \quad\text{quasi-everywhere in }B.
\]
Let \(E_1\) and \(E_2\) be the exceptional sets in the two quasi-everywhere
statements
\[
 \tilde u=0
 \quad\text{on } \R^N\setminus\Omega
\]
and
\[
 \tilde u=c_0\phi(|x-a|)
 \quad\text{in }B,
\]
respectively.  Then \(\SobCap(E_1)=\SobCap(E_2)=0\).  If
\(x\in (B\setminus\Omega)\setminus(E_1\cup E_2)\), then the first statement
gives \(\tilde u(x)=0\), while the second gives $\tilde u(x)=c_0\phi(|x-a|)>0$, because \(c_0>0\) and \(\phi\) is strictly positive in \(B\).  This is
impossible. Hence $B\setminus\Omega\subset E_1\cup E_2$. By monotonicity and subadditivity of the Sobolev capacity, we have
\[
 \SobCap(B\setminus\Omega)=0.
\]

We next show that $\Omega\cap\partial B=\varnothing$. Suppose instead that \(x_0\in\Omega\cap\partial B\).  Since \(\Omega\) is
open, one can choose \(\zeta\in C_c^\infty(\Omega)\), \(\zeta\ge0\), with
\(\zeta(x_0)>0\).  In particular, after shrinking the support if necessary,
\[
 \int_{\partial B}\zeta\,\dd\mathcal H^{N-1}>0.
\]
Since \(u=c_0\phi(|\cdot-a|)\chi_B\) a.e. in \(\R^N\), and since
\(\SobCap(B\setminus\Omega)=0\) implies \(|B\setminus\Omega|=0\), the weak
eigenvalue equation tested against \(\zeta\) gives
\[
 0
 =
 \int_\Omega \nabla u\cdot\nabla\zeta\,\dd x
 -\alpha^2\int_\Omega u\zeta\,\dd x
 =
 c_0\left(
 \int_B \nabla\phi(|x-a|)\cdot\nabla\zeta\,\dd x
 -\alpha^2\int_B \phi(|x-a|)\zeta\,\dd x
 \right).
\]
Integrating by parts in \(B\), and using
\(-\Delta\phi(|x-a|)=\alpha^2\phi(|x-a|)\) in \(B\), we obtain
\[
 0
 =
 c_0\int_{\partial B}\zeta\,\partial_\nu\phi\,\dd\mathcal H^{N-1}
 =
 c_0\phi'(1)
 \int_{\partial B}\zeta\,\dd\mathcal H^{N-1},
\]
where \(\partial_\nu\) denotes the outward normal derivative on \(\partial B\).
Since \(j_{\nu,1}\) is a simple zero of \(J_\nu\), one has
\(\phi'(1)=\alpha J_\nu'(\alpha)\ne0\).  Together with
\(\phi>0\) on \([0,1)\) and \(\phi(1)=0\), this forces
\(\phi'(1)<0\). Hence
\[
 \partial_\nu\phi=\phi'(1)<0
 \quad\text{on }\partial B.
\]
This contradicts \(c_0>0\) and
\(\int_{\partial B}\zeta\,\dd\mathcal H^{N-1}>0\). Therefore $\Omega\cap\partial B=\varnothing$.

Since sets of zero \(H^1\)-capacity have Lebesgue measure zero and \(B\) has
positive measure, the identity \(\SobCap(B\setminus\Omega)=0\) implies that
\(\Omega\cap B\ne\varnothing\).  

We claim that $\Omega\setminus\overline B=\varnothing$. Indeed, if \(\Omega\setminus\overline B\) were nonempty, then, because
\(\Omega\cap\partial B=\varnothing\), we would have the disjoint decomposition $\Omega=(\Omega\cap B)\,\dot\cup\,(\Omega\setminus\overline B)$. Both sets on the right are nonempty and open in \(\Omega\), contradicting the
connectedness of the domain \(\Omega\). Hence $\Omega\setminus\overline B=\varnothing$. Together with \(\Omega\cap\partial B=\varnothing\), this gives
\(\Omega\subset B\). Therefore $\Omega\triangle B=B\setminus\Omega$, and hence
\[
 \SobCap(\Omega\triangle B)=0
\]
in the normalized variables.

It remains to return to the original scale. Let \(\Omega_{\mathrm{orig}}\) denote the domain before the normalization in the equality argument, and set $s=\sqrt{\lambda_1(\Omega_{\mathrm{orig}})}/\alpha$. Then \(\lambda_1(s\Omega_{\mathrm{orig}})=\alpha^2\). The normalized equality
analysis gives, for some \(a\in\R^N\),
\[
 \SobCap\bigl((s\Omega_{\mathrm{orig}})\triangle B_1(a)\bigr)=0.
\]
Applying the inverse dilation \(x\mapsto x/s\), and using $s^{-1}B_1(a)=B_{1/s}(a/s)$, we obtain
\[
 \SobCap\bigl(\Omega_{\mathrm{orig}}\triangle B_{1/s}(a/s)\bigr)=0,
\]
because the property of having zero Sobolev \(H^1\)-capacity is preserved
under dilations.  This is precisely \eqref{eq:equality-capacity} for the
original domain.

It remains to show that the capacitary equality becomes literal equality under
the Lipschitz assumption.  In the normalized variables, suppose that
\(\Omega\) is a Lipschitz domain.  We have already proved $\Omega\triangle B=B\setminus\Omega$. Suppose, for contradiction, that \(\Omega\ne B\).  Choose
\(x_0\in B\setminus\Omega\).  If \(x_0\in B\setminus\overline\Omega\), then
\(B\setminus\Omega\) contains a nonempty open ball and hence has positive
\(H^1\)-capacity. Otherwise \(x_0\in\partial\Omega\cap B\).  Since \(\Omega\) is a Lipschitz
domain, its boundary is locally the graph of a Lipschitz function; see
\cite[Definition~1.2.1.1]{Grisvard2011}.  The elementary cone construction
from this local graph representation gives an exterior truncated cone with
vertex \(x_0\) contained in \(\R^N\setminus\Omega\).  Since \(x_0\in B\), we
may shorten the cone, if necessary, so that it is contained in
\(B\setminus\Omega\).  This cone contains a nonempty open ball; since sets of
zero \(H^1\)-capacity have Lebesgue measure zero, it has positive
\(H^1\)-capacity. Thus, in either case, $\SobCap(B\setminus\Omega)>0$. This contradicts
\[
 \SobCap(\Omega\triangle B)=\SobCap(B\setminus\Omega)=0.
\]
Therefore \(\Omega=B\) in the normalized variables.  Scaling back, the
original Lipschitz equality case is a Euclidean ball.  This completes the
proof of \cref{thm:main}.

\medskip

We record the corresponding variant for bounded open sets which are not assumed to be connected.

\begin{remark}\label{rem:disconnected}
The inequality in \cref{thm:main} also extends to arbitrary bounded open
sets, with the convention that any term with
\(\lambda_{i+1}(\Omega)=\lambda_1(\Omega)\) is interpreted as \(+\infty\).
Indeed, if \(\lambda_1(\Omega)\) is not simple, then
\(\lambda_2(\Omega)=\lambda_1(\Omega)\), and the left-hand side of
\eqref{eq:main} is \(+\infty\). If \(\lambda_1(\Omega)\) is simple, the proof of the inequality above applies
with the same argument to a normalized first eigenfunction \(u\).  The
connectedness assumption was used only to ensure the simplicity of
\(\lambda_1\) and, later, in the equality analysis.  In the trace Ritz step,
the needed spectral input is precisely the positive gap
\(\lambda_2(\Omega)-\lambda_1(\Omega)>0\) on \(u^\perp\).

The equality case is correspondingly larger.  With the convention
\(\lambda_1(\varnothing)=+\infty\), equality holds precisely for those
bounded open sets \(\Omega\) for which there exist a ball \(B_R(a)\) and a
possibly empty bounded open set \(D\), with
\(D\cap\overline{B_R(a)}=\varnothing\), such that
\[
 \SobCap\bigl(\Omega\triangle (B_R(a)\,\dot\cup\,D)\bigr)=0
\qquad\text{and}\qquad
 \lambda_1(D)\ge\lambda_2(B_R(a)).
\]
The sufficiency is immediate from the fact that the Dirichlet spectrum of a
disjoint union is the ordered union of the spectra of its components, counted
with multiplicities.  Under the displayed condition on \(D\), the first
\(N\) eigenvalues after \(\lambda_1(B_R(a))\) are still equal to \(\lambda_2(B_R(a))\), and hence the reciprocal
sum has the ball value.

Conversely, suppose equality holds for an arbitrary bounded open set.  The
case of nonsimple \(\lambda_1\) is excluded by the \(+\infty\) convention.
Repeating the equality analysis above, up to but not including the final
connectedness argument, gives a ball \(B\) such that
\[
 \SobCap(B\setminus\Omega)=0,
 \qquad
 \Omega\cap\partial B=\varnothing.
\]
Set $D=\Omega\setminus\overline B$. Then \(D\) is a possibly empty bounded open set and
\(D\cap\overline B=\varnothing\).  Since \(\Omega\cap\partial B=\varnothing\),
we have $\Omega\triangle(B\,\dot\cup\,D)=B\setminus\Omega$. Hence
\[
 \SobCap\bigl(\Omega\triangle(B\,\dot\cup\,D)\bigr)=0.
\]
Capacity-zero modifications do not change \(H_0^1\), and therefore do not
change the Dirichlet spectrum.  Thus \(\Omega\) and \(B\,\dot\cup\,D\) have
the same spectrum.

If \(D\ne\varnothing\), then necessarily $\lambda_1(D)>\lambda_1(B)$. Indeed, if \(\lambda_1(D)<\lambda_1(B)\), then
\[
 \lambda_1(B\,\dot\cup\,D)=\lambda_1(D)<\lambda_1(B)=\lambda_1(\Omega)=\lambda_1(B\,\dot\cup\,D),
\]
a contradiction.  If \(\lambda_1(D)=\lambda_1(B)\), then the first
eigenvalue of \(B\,\dot\cup\,D\) has multiplicity at least two, one
eigenfunction being supported on \(B\) and another on \(D\).  This contradicts
the already excluded nonsimple case.  Therefore $\lambda_1(D)>\lambda_1(B)$.

We have already shown that, if \(D\ne\varnothing\), then
\(\lambda_1(D)>\lambda_1(B)\).  Suppose now, for contradiction, that $\lambda_1(D)<\lambda_2(B)$. Then
\[
 \lambda_1(B)<\lambda_1(D)<\lambda_2(B).
\]
In the ordered spectrum of \(B\,\dot\cup\,D\), the eigenvalue
\(\lambda_1(D)\) therefore appears among the eigenvalues above
\(\lambda_1(B)\) and below \(\lambda_2(B)\).  On the other hand, the ball
contributes \(N\) eigenvalues equal to \(\lambda_2(B)\).  Hence the first
\(N\) eigenvalues of \(B\,\dot\cup\,D\) above \(\lambda_1(B)\) are all at
most \(\lambda_2(B)\), and at least one of them is strictly smaller.
Consequently, the corresponding \(N\) gaps are all at most \(\lambda_2(B)-\lambda_1(B)\), and at least
one is strictly smaller.  Therefore
\[
 \sum_{i=1}^N
 \frac{\lambda_1(B)}
 {\lambda_{i+1}(B\,\dot\cup\,D)-\lambda_1(B)}
 >
 \frac{N\lambda_1(B)}{\lambda_2(B)-\lambda_1(B)}.
\]
This is strictly larger than the ball value, contradicting equality.  Hence
\[
 \lambda_1(D)\ge\lambda_2(B).
\]
Scaling back gives the stated condition for \(B_R(a)\).  This explains why
connectedness is the natural assumption for the sharp equality statement in
\cref{thm:main}.
\end{remark}

\appendix

\section{Algebra for the Riccati barrier residuals}\label{app:barrier-algebra}

The elementary algebra leading to the residual formulas used in
\cref{subsec:riccati-barriers} is collected here.

\begin{lemma}\label{lem:barrier-algebra}
Let
\[
 \mathcal T(x)=\sum_{i=2}^{\infty}\frac{y_i}{1+y_ix^2},
 \qquad
 x=\sqrt{1-r^2}.
\]
For the functions
\[
 Q_-(r)=x^2,
 \qquad
 Q_{\kappa}(r)=\frac{\kappa x^2}{1+(\kappa-1)x^2},
 \quad \kappa>1,
\]
the residual \(E_Q\) defined in \eqref{eq:EQ} satisfies
\eqref{eq:lower-residual}--\eqref{eq:Phi-minus-derivative} and
\eqref{eq:upper-residual}--\eqref{eq:A0-A1}, respectively.
\end{lemma}

\begin{proof}
Write
\[
 R=1-x^2=r^2,
 \qquad
 D_{\kappa}=1+(\kappa-1)x^2.
\]
If \(Q=Q(x)\), then \(dx/dr=-r/x\), and hence
\[
 r\frac{dQ}{dr}=-\frac{R}{x}Q_x.
\]
Using \eqref{eq:h-x}, and writing \(\mathcal T=\mathcal T(x)\), the residual \eqref{eq:EQ}
becomes
\begin{equation}\label{eq:app-general-residual}
 E_Q=-\frac{R}{x}Q_x+Q^2+(N-2)Q
      -4R\frac{Q}{x^2}-4RQ\mathcal T+CR-(N-1).
\end{equation}
We shall apply this identity only to the two barriers above, for which
\(Q/x^2\) is regular at \(x=0\).

For $Q_-=x^2$, \eqref{eq:app-general-residual} gives
\begin{align*}
 E_{Q_-}
 &=-2R+x^4+(N-2)x^2-4R-4Rx^2\mathcal T+CR-(N-1)\\
 &=R\left(C-N-5-x^2-4x^2\mathcal T\right).
\end{align*}
This is \eqref{eq:lower-residual}--\eqref{eq:Phi-minus}.  Since
\[
 \frac{d}{dx}\left(\frac{x^2y}{1+yx^2}\right)
 =\frac{2xy}{(1+yx^2)^2},
\]
termwise differentiation, justified by the uniform convergence stated after
\eqref{eq:sum-y}, gives \eqref{eq:Phi-minus-derivative}.

For the upper barrier, \(Q_{\kappa}=\kappa x^2/D_{\kappa}\), and
\begin{equation}\label{eq:app-QK-prime}
 Q_{\kappa}'(x)=\frac{2\kappa x}{D_{\kappa}^2}.
\end{equation}
Substituting \(Q_{\kappa}\) into \eqref{eq:app-general-residual} yields
\begin{align*}
 E_{Q_{\kappa}}
 ={}&R(C-4Q_{\kappa}\mathcal T)
 -\frac{2\kappa R}{D_{\kappa}^2}+\frac{\kappa^2x^4}{D_{\kappa}^2}
 +(N-2)\frac{\kappa x^2}{D_{\kappa}}-4\frac{\kappa R}{D_{\kappa}}-(N-1).
\end{align*}
Putting the non-series terms after \(R(C-4Q_{\kappa}\mathcal T)\) over the common
denominator \(D_{\kappa}^2\), and collecting powers of \(x^2\), gives
\[
 E_{Q_{\kappa}}
 =
 R\left(C+\frac{B_0+B_1x^2}{D_{\kappa}^2}-4Q_{\kappa}\mathcal T\right),
\]
where
\[
 B_0=-6\kappa-N+1,
 \qquad
 B_1=-4\kappa^2-\kappa N+4\kappa+N-1.
\]
This is \eqref{eq:upper-residual}--\eqref{eq:Phi-K}.

It remains to compute the derivative. Put $\mathcal P(x)=B_0+B_1x^2$. Then
\[
 \left(\frac{\mathcal P}{D_{\kappa}^2}\right)'
 =
 \frac{2x}{D_{\kappa}^3}
 \left(B_1D_{\kappa}-2(\kappa-1)\mathcal P\right).
\]
A direct collection of coefficients gives
\[
 B_1D_{\kappa}-2(\kappa-1)\mathcal P=A_0+A_1x^2+N\kappa D_{\kappa},
\]
where
\[
 A_0=8\kappa^2-10\kappa-N+1,
 \qquad
 A_1=4\kappa^3-8\kappa^2-\kappa N+5\kappa+N-1.
\]
The term \(N\kappa D_{\kappa}\) is separated in order to use the Bessel-zero sum identity
\eqref{eq:sum-y}.  Indeed, by the formula for \(Q_{\kappa}'\) in \eqref{eq:app-QK-prime}, its contribution to
the derivative is
\[
 \frac{2xN\kappa D_{\kappa}}{D_{\kappa}^3}=NQ_{\kappa}'(x)
 =
 4\sum_{i=2}^{\infty}Q_{\kappa}'(x)y_i.
\]
Therefore, with termwise differentiation justified by the uniform convergence
stated after \eqref{eq:sum-y},
\begin{align*}
 \Phi_{\kappa}'(x)
 &=\frac{2x(A_0+A_1x^2)}{D_{\kappa}^3}
   +4\sum_{i=2}^{\infty}Q_{\kappa}'(x)y_i
   -4\sum_{i=2}^{\infty}
    \left(\frac{Q_{\kappa}(x)y_i}{1+y_i x^2}\right)' \\
 &=\frac{2x(A_0+A_1x^2)}{D_{\kappa}^3}
   +4\sum_{i=2}^{\infty}
    \left[
      Q_{\kappa}'(x)y_i-
      \left(\frac{Q_{\kappa}(x)y_i}{1+y_ix^2}\right)'
    \right],
\end{align*}
which is \eqref{eq:Phi-K-derivative}.
\end{proof}

\section*{Acknowledgments and AI disclosure}

During the development and preparation of this work, the authors used a
generative AI tool for preliminary, non-authoritative exploratory assistance,
including organizational discussion and the consideration of possible
approaches.  The AI-generated outputs were not treated as mathematical
sources, and no argument was included without independent verification and
substantial revision by the authors.  All theorem statements, proofs,
computations, references, and final arguments were independently checked,
revised, and finalized by the authors, who take full responsibility for the
correctness, originality, and integrity of the paper.


\begin{thebibliography}{99}

\bibitem{Ashbaugh1999}
M.~S. Ashbaugh,
\emph{Open problems on eigenvalues of the Laplacian},
in \emph{Analytic and Geometric Inequalities and Applications}, Math. Appl., vol.~478,
Kluwer Academic Publishers, Dordrecht, 1999, pp.~13--28.

\bibitem{Ashbaugh2002}
M.~S. Ashbaugh,
\emph{The universal eigenvalue bounds of Payne--P\'olya--Weinberger, Hile--Protter, and H.~C. Yang},
Proc. Indian Acad. Sci. Math. Sci. \textbf{112} (2002), no.~1, 3--30.

\bibitem{AshbaughBenguria1991}
M.~S. Ashbaugh and R.~D. Benguria,
\emph{Proof of the Payne--P\'olya--Weinberger conjecture},
Bull. Amer. Math. Soc. (N.S.) \textbf{25} (1991), no.~1, 19--29.

\bibitem{AshbaughBenguria1992a}
M.~S. Ashbaugh and R.~D. Benguria,
\emph{A sharp bound for the ratio of the first two eigenvalues of Dirichlet Laplacians and extensions},
Ann. of Math. (2) \textbf{135} (1992), no.~3, 601--628.

\bibitem{AshbaughBenguria1992b}
M.~S. Ashbaugh and R.~D. Benguria,
\emph{A second proof of the Payne--P\'olya--Weinberger conjecture},
Comm. Math. Phys. \textbf{147} (1992), no.~1, 181--190.

\bibitem{AshbaughBenguria1993}
M.~S. Ashbaugh and R.~D. Benguria,
\emph{More bounds on eigenvalue ratios for Dirichlet Laplacians in $n$ dimensions},
SIAM J. Math. Anal. \textbf{24} (1993), no.~6, 1622--1651.

\bibitem{AshbaughBenguria1993Neumann}
M.~S. Ashbaugh and R.~D. Benguria,
\emph{Universal bounds for the low eigenvalues of Neumann Laplacians in $n$ dimensions},
SIAM J. Math. Anal. \textbf{24} (1993), no.~3, 557--570.

\bibitem{AshbaughBenguria1994}
M.~S. Ashbaugh and R.~D. Benguria,
\emph{Bounds for ratios of eigenvalues of the Dirichlet Laplacian},
Proc. Amer. Math. Soc. \textbf{121} (1994), no.~1, 145--150.

\bibitem{AshbaughBenguria2001}
M.~S. Ashbaugh and R.~D. Benguria,
\emph{A sharp bound for the ratio of the first two Dirichlet eigenvalues of a domain in a hemisphere of $S^n$},
Trans. Amer. Math. Soc. \textbf{353} (2001), no.~3, 1055--1087.

\bibitem{AshbaughHermi2004}
M.~S. Ashbaugh and L.~Hermi,
\emph{A unified approach to universal inequalities for eigenvalues of elliptic operators},
Pacific J. Math. \textbf{217} (2004), no.~2, 201--219.

\bibitem{BariczJankovMasirevicPoganySzasz2015}
\'A.~Baricz, D.~Jankov Ma\v{s}irevi\'c, T.~K. Pog\'any and R.~Sz\'asz,
\emph{On an identity for zeros of Bessel functions},
J. Math. Anal. Appl. \textbf{422} (2015), no.~1, 27--36.

\bibitem{BenguriaLinde2007}
R.~D. Benguria and H.~Linde,
\emph{A second eigenvalue bound for the Dirichlet Laplacian in hyperbolic space},
Duke Math. J. \textbf{140} (2007), no.~2, 245--279.

\bibitem{BenguriaLindeLoewe2012}
R.~D. Benguria, H.~Linde and B.~Loewe,
\emph{Isoperimetric inequalities for eigenvalues of the Laplacian and the Schr\"odinger operator},
Bull. Math. Sci. \textbf{2} (2012), no.~1, 1--56.

\bibitem{Brands1964}
J.~J.~A.~M. Brands,
\emph{Bounds for the ratios of the first three membrane eigenvalues},
Arch. Rational Mech. Anal. \textbf{16} (1964), 265--268.

\bibitem{Brezis2011}
H.~Brezis,
\emph{Functional Analysis, Sobolev Spaces and Partial Differential Equations},
Universitext, Springer, New York, 2011.

\bibitem{BucurButtazzo2005}
D.~Bucur and G.~Buttazzo,
\emph{Variational Methods in Shape Optimization Problems},
Progress in Nonlinear Differential Equations and Their Applications, vol.~65,
Birkh\"auser Boston, Boston, MA, 2005.

\bibitem{ChenMao2024}
R.~Chen and J.~Mao,
\emph{On the Ashbaugh--Benguria type conjecture about lower-order Neumann eigenvalues of the Witten-Laplacian},
arXiv:2403.08070, 2024.

\bibitem{ChenZheng2011}
Z.-C. Chen and T.~Zheng,
\emph{Bounds for ratios of the membrane eigenvalues},
J. Differential Equations \textbf{250} (2011), no.~3, 1575--1590.

\bibitem{Chiti1982}
G.~Chiti,
\emph{A reverse H\"older inequality for the eigenfunctions of linear second order elliptic operators},
Z. Angew. Math. Phys. \textbf{33} (1982), no.~1, 143--148.

\bibitem{Chiti1983}
G.~Chiti,
\emph{A bound for the ratio of the first two eigenvalues of a membrane},
SIAM J. Math. Anal. \textbf{14} (1983), no.~6, 1163--1167.

\bibitem{EvansGariepy2015}
L.~C. Evans and R.~F. Gariepy,
\emph{Measure Theory and Fine Properties of Functions}, revised ed.,
CRC Press, Boca Raton, FL, 2015.

\bibitem{Faber1923}
G.~Faber,
\emph{Beweis, da\ss{} unter allen homogenen Membranen von gleicher Fl\"ache und gleicher Spannung die kreisf\"ormige den tiefsten Grundton gibt},
Sitzungsberichte der mathematisch-physikalischen Klasse der Bayerischen Akademie der Wissenschaften, 1923, pp.~169--172.

\bibitem{Grisvard2011}
P.~Grisvard,
\emph{Elliptic Problems in Nonsmooth Domains},
Classics in Applied Mathematics, vol.~69,
SIAM, Philadelphia, 2011.

\bibitem{HeLiTang2026}
Y.~He, Y.~Li and Q.~Tang,
\emph{A proof of the Ashbaugh--Benguria conjecture for reciprocal sums of Neumann eigenvalues},
arXiv:2606.08271, 2026.

\bibitem{Henrot2006}
A.~Henrot,
\emph{Extremum Problems for Eigenvalues of Elliptic Operators},
Frontiers in Mathematics, Birkh\"auser, Basel, 2006.

\bibitem{Hersch1961}
J.~Hersch,
\emph{Caract\'erisation variationnelle d'une somme de valeurs propres cons\'ecutives; g\'en\'eralisation d'in\'egalit\'es de P\'olya--Schiffer et de Weyl},
C. R. Acad. Sci. Paris \textbf{252} (1961), 1714--1716.

\bibitem{HileProtter1980}
G.~N. Hile and M.~H. Protter,
\emph{Inequalities for eigenvalues of the Laplacian},
Indiana Univ. Math. J. \textbf{29} (1980), no.~4, 523--538.

\bibitem{HileXu1993}
G.~N. Hile and Z.~Y. Xu,
\emph{Inequalities for sums of reciprocals of eigenvalues},
J. Math. Anal. Appl. \textbf{180} (1993), no.~2, 412--430.

\bibitem{Krahn1925}
E.~Krahn,
\emph{\"Uber eine von Rayleigh formulierte Minimaleigenschaft des Kreises},
Math. Ann. \textbf{94} (1925), 97--100.

\bibitem{Krahn1926}
E.~Krahn,
\emph{\"Uber Minimaleigenschaften der Kugel in drei und mehr Dimensionen},
Acta Comm. Univ. Tartu (Dorpat) A \textbf{9} (1926), 1--44.

\bibitem{Laugesen2021}
R.~S. Laugesen,
\emph{Well-posedness of Weinberger's center of mass by Euclidean energy minimization},
J. Geom. Anal. \textbf{31} (2021), no.~9, 8762--8779.

\bibitem{LevitinMangoubiPolterovich2023}
M.~Levitin, D.~Mangoubi and I.~Polterovich,
\emph{Topics in Spectral Geometry},
Graduate Studies in Mathematics, vol.~237, American Mathematical Society, Providence, RI, 2023.

\bibitem{Marcellini1980}
P.~Marcellini,
\emph{Bounds for the third membrane eigenvalue},
J. Differential Equations \textbf{37} (1980), no.~3, 438--443.

\bibitem{DLMF}
NIST Digital Library of Mathematical Functions, \emph{Chapter 10: Bessel Functions}, especially \S10.21, \url{https://dlmf.nist.gov/10.21}.

\bibitem{PaynePolyaWeinberger1955}
L.~E. Payne, G.~P\'olya and H.~F. Weinberger,
\emph{Sur le quotient de deux fr\'equences propres cons\'ecutives},
C. R. Acad. Sci. Paris \textbf{241} (1955), 917--919.

\bibitem{PaynePolyaWeinberger1956}
L.~E. Payne, G.~P\'olya and H.~F. Weinberger,
\emph{On the ratio of consecutive eigenvalues},
J. Math. and Phys. \textbf{35} (1956), 289--298.

\bibitem{Szego1954}
G.~Szeg\H{o},
\emph{Inequalities for certain eigenvalues of a membrane of given area},
J. Rational Mech. Anal. \textbf{3} (1954), 343--356.

\bibitem{Thompson1969}
C.~J. Thompson,
\emph{On the ratio of consecutive eigenvalues in $n$-dimensions},
Stud. Appl. Math. \textbf{48} (1969), 281--283.

\bibitem{WangXia2021}
Q.~Wang and C.~Xia,
\emph{On Ashbaugh--Benguria's conjecture about lower-order Dirichlet eigenvalues of the Laplacian},
Anal. PDE \textbf{14} (2021), no.~7, 2069--2078.


\bibitem{Weinberger1956}
H.~F. Weinberger,
\emph{An isoperimetric inequality for the $N$-dimensional free membrane problem},
J. Rational Mech. Anal. \textbf{5} (1956), 633--636.

\bibitem{XiaWang2023}
C.~Xia and Q.~Wang,
\emph{On a conjecture of Ashbaugh and Benguria about lower eigenvalues of the Neumann Laplacian},
Math. Ann. \textbf{385} (2023), no.~1--2, 863--879.

\end{thebibliography}
\end{document}